\PassOptionsToPackage{dvipsnames}{xcolor} 
\documentclass[preprint,11pt]{article}
\usepackage{geometry}
\usepackage{graphicx}
\usepackage[utf8]{inputenc}
\usepackage{multirow} 
\usepackage{amssymb,amsmath} 
\usepackage{tikz}
\usepackage{amsthm}
\usepackage{authblk}

\usetikzlibrary{shapes,arrows,backgrounds,fit,positioning}
\usetikzlibrary{chains, decorations.pathreplacing, positioning}
\usetikzlibrary{calc}
\usetikzlibrary{arrows.meta}
\usetikzlibrary{shapes,snakes}
\usetikzlibrary{decorations.pathmorphing}
\usetikzlibrary{decorations.markings}
\usetikzlibrary{calc}
\usetikzlibrary{backgrounds}
\usetikzlibrary{shadows}

\newtheorem{proposition}{Proposition}

\newtheorem{lemma}{Lemma}
\newtheorem{theorem}{Theorem}
\newtheorem{definition}{Definition}
\newtheorem{remark}{Remark}
\newtheorem{example}{Example}

\newcommand{\bool}{\{0,1\}}
\newcommand{\M}{\boldsymbol{M}}
\newcommand{\MU}{\boldsymbol{\mu}}
\newcommand{\LAMBDA}{\boldsymbol{\lambda}}
\newcommand{\GLAMBDA}{\boldsymbol{\Lambda}}
\newcommand{\ELL}{\boldsymbol{\ell}}
\newcommand{\C}{\boldsymbol{c}}

\newcommand{\glambda}{\Lambda}
\newcommand{\functions}{\mathrm{F}}
\newcommand{\xor}{{\oplus}}
\newcommand{\ie}{\emph{i.e.}}

\newcommand{\maxpf}{\phi^{\max}}
\newcommand{\minpf}{\phi^{\min}}

\newcommand{\KMaxFPP}{{\textsc{$k$-Maximum Fixed Point Problem}}}
\newcommand{\kmaxfpp}[1]{$#1$-{\textsc{MaxFPP}}}
\newcommand{\MaxFPP}{{\textsc{Maximum Fixed Point Problem}}}
\newcommand{\maxfpp}{{\textsc{MaxFPP}}}

\newcommand{\KMinFPP}{{\textsc{$k$-Minimum Fixed Point Problem}}}
\newcommand{\kminfpp}[1]{$#1$-{\textsc{MinFPP}}}
\newcommand{\MinFPP}{{\textsc{Minimum Fixed Point Problem}}}
\newcommand{\minfpp}{{\textsc{MinFPP}}}

\newcommand{\SSAT}{{\textsc{Succinct-3SAT}}}
\newcommand{\SthreeSAT}{{\textsc{Succinct-3SAT}}}

\newcommand{\sharpSAT}{{\textsc{\#SAT}}}

\newcommand{\EMSAT}{{\textsc{Existential-Majority-SAT}}}
\newcommand{\emsat}{{\textsc{E-MajSAT}}}
\newcommand{\MAJSAT}{{\textsc{Majority-SAT}}}

\newcommand{\QSAT}{\textsc{Quantified satisfiability} with 2 alternating quantifiers}
\newcommand{\qsat}{\textsc{QSAT$_2$}}

\newcommand{\NEXPTIME}{\mathrm{\sf NEXPTIME}}
\newcommand{\Poly}{\mathrm{\sf P}}
\newcommand{\PP}{\mathrm{\sf PP}}
\newcommand{\SPoly}{\mathrm{\sf \#P}}
\newcommand{\NP}{\mathrm{\sf NP}}
\newcommand{\coNP}{\mathrm{\sf coNP}}
\newcommand{\NPoSPoly}{\NP^{\SPoly}}
\newcommand{\NPoPP}{\NP^{\PP}}
\newcommand{\PolyoSPoly}{\Poly^{\SPoly}}
\newcommand{\PolyoPP}{\Poly^{\PP}}
\newcommand{\NPoNP}{\NP^{\NP}}

\newcommand{\an}{a}
\newcommand{\simple}{nice}
\newcommand{\Null}{zero}

\newcommand{\decisionpb}[3]{\fbox{\parbox{0.87\textwidth}{{#1}\\{\it Input:} #2\\{\it Question:} #3}}}

\title{
  Complexity of fixed point counting problems\\
  in Boolean Networks\thanks{This work extends some early results
  announced in \cite{bdpr19}.}
}

\author[1]{Florian Bridoux}
\author[1]{Amélia Durbec}
\author[1,2]{Kevin Perrot}
\author[2,3]{Adrien Richard}
\affil[1]{Aix-Marseille Universit\'e, Universit\'e de Toulon, CNRS, LIS, Marseille, France}
\affil[2]{Universit\'e Côte d’Azur, CNRS, I3S, Sophia Antipolis, France}
\affil[3]{Universidad de Chile, CMM, Santiago, Chile}

\date{\today}

\begin{document}

\maketitle
	
\begin{abstract}
A {\em Boolean network} (BN) with $n$ components is a discrete dynamical system described by the successive iterations of a function $f:\bool^n \to \bool^n$. This model finds applications in biology, where fixed points play a central role. For example, in genetic regulations, they correspond to cell phenotypes. In this  context, experiments reveal the existence of positive or  negative influences among components. The digraph of influences is called {\em signed  interaction digraph} (SID), and one SID may correspond to a large number of BNs. The present work opens a new perspective on the well-established study of fixed points in BNs. When biologists discover the SID of a BN they do not know, they may ask: given that SID, can it correspond to a BN having at least/at most $k$ fixed points? Depending on the input, we prove that these problems are in $\Poly$ or complete for $\NP$, $\NPoNP$, $\NPoSPoly$ or $\NEXPTIME$.
	
\end{abstract}

\section{Introduction}\label{section:introduction}

A {\em Boolean network} (BN) with $n$ components is a discrete dynamical system described by the successive iterations of a function
\[
f:\bool^n\to\bool^n,\qquad x=(x_1,\dots,x_n)\mapsto 
f(x)=(f_1(x),\dots,f_n(x)).
\]
The structure of the network is often described by a signed digraph $D$, called {\em signed interaction digraph} (SID) of $f$, catching effective positive and negative dependencies among components: the vertex set is $[n]=\{1,\dots,n\}$ and, for all $i,j\in [n]$, there is a positive (resp. negative, \Null{}) arc from $j$ to $i$ if $f_i(x)$ depends essentially on $x_j$ and is an increasing (resp. decreasing, non-monotone) function of $x_j$. The SID provides a very rough information about~$f$ since, given \an{} SID $D$, the set $\functions(D)$ of BNs whose SID is $D$ is generally huge. 

\medskip
BNs have many applications. In particular, since the seminal papers of Kauffman \cite{K69,K93} and Thomas \cite{T73,TA90}, they are very classical models for the dynamics of gene networks. In this context, the first reliable experimental information often concern the SID of the network, while the actual dynamics are very difficult to observe \cite{TK01,N15}. One is thus faced with the following question: {\em What can be said about the dynamics described by $f$ according to $D$ only?}  

\medskip
Among the many dynamical properties that can be studied, fixed points are of special interest. For instance, in the context of gene networks, they correspond to stable patterns of gene expression at the basis of particular cellular phenotypes \cite{TA90,A04}. Furthermore, from a theoretical point of view, fixed points are among the very few properties invariant under the classical update schedules ({\em e.g} parallel, sequential, asynchronous) that we can use to define a dynamics from $f$ \cite{AGMS09}. As such, they are arguably the property which has been the most thoroughly studied. The number of fixed points and its maximization in particular is the subject of a stream of work, {\em e.g.} in \cite{R86,ADG04b,RRT08,A08,GR11,ARS14,GRR15,ARS17}. 

\medskip
From the computational complexity point of view, previous works essentially focused on decision problems of the following form: given $f$ and a dynamical property $P$, what is the complexity of deciding if the dynamics described by $f$ has the property $P$. For instance, it is well-known that deciding if $f$ has a fixed point is $\NP$-complete in general (see \cite{K08} and the references therein), and in $\Poly$ for some families of BNs, such as non-expansive BNs \cite{F92}. However, as mentioned above, in practice, $f$ is often unknown while its SID is well approximated. Hence, a more natural question is: given \an{} SID $D$ and a dynamical property $P$, what is the complexity of deciding if the dynamics described by some $f\in \functions(D)$ has the property $P$. To the best of our knowledge, there is, perhaps surprisingly, almost no work concerning this kind of decision problems; we can however cite \cite{dasgupta2007algorithmic} which, in the continuous setting, studies the complexity of an optimization problem based on the SID and dynamical properties related to monotonicity.

\medskip
In this paper, we study this class of decision problems, focusing on the maximum and minimum number of fixed points. More precisely, given \an{} SID $D$, we respectively denote by $\maxpf(D)$ and $\minpf(D)$ the maximum and minimum number of fixed points in a BN $f\in \functions(D)$, and we study the complexity of deciding if $\maxpf(D)\geq k$ or $\minpf(D) < k$.

\subsection{Organization of the paper and summary of contributions} 

After the preliminaries given in Section~\ref{section:definition}, we first study the problem of deciding if $\maxpf(D)\geq k$  when the positive integer $k$ is fixed. In Section~\ref{section:max_k=1} we prove that this problem is in $\Poly$ when $k=1$ and, in Section~\ref{section:max_k>=2}, we prove that it is $\NP$-complete when $k\geq 2$. Furthermore, these results remain true if the maximum in-degree $\Delta(D)$ is bounded by any constant $d\geq 2$. The case $k=2$ is of particular interest since many works have been devoted to finding necessary conditions for the existence of multiple fixed points, both in the discrete and continuous settings, see \cite{S06,RRT08,S06,KST07,R19} and the references therein. Section~\ref{section:max_k>=k} considers the case where $k$ is part of the input. We prove that, given \an{} SID $D$ and a positive integer $k$, deciding if $\maxpf(D)\geq k$ is $\NEXPTIME$-complete, and becomes $\NPoSPoly$-complete if $\Delta(D)$ is bounded by a constant $d\geq 2$. We then study the minimum number of fixed points, in Section~\ref{section:min}. We prove that, even for $k = 1$, deciding if $\minpf(D) < k$ is $\NEXPTIME$-complete. When $\Delta(D)$ is bounded by a constant $d\geq 2$, it becomes $\NPoNP$-complete if $k$ is a constant, and $\NPoSPoly$-complete if $k$ is a parameter of the problem. Finally, a conclusion and some perspectives are given in Section~\ref{sec:conclu}. 

\medskip
A summary of the results is given in Table~\ref{table:classe_complexite}. Note that, from them, we immediately obtain complexity results for the dual decision problems $\maxpf(D)<k$ and  $\minpf(D) \geq k$. 

\begin{table}[ht]
	\centerline{
		\setlength{\tabcolsep}{0.5em}
		\def\arraystretch{1.3}
		\begin{tabular}{|c|c|c|c|c|}
			\hline 
			Problem  & $\Delta(D) \leq d$ & $k = 1$ & $k \geq 2$ & $k$ given in input\\
			\hline 
			\multirow{2}{*}{$ \maxpf(D)  \geq k $}  
			& no & \multirow{2}{*}{$\Poly$} & \multirow{2}{*}{$\NP$-complete} & $\NEXPTIME$-complete\\
			\cline{2-2}
			\cline{5-5}
			& yes &  &  & $\NPoSPoly$-complete\\
			\cline{1-2}
			\cline{4-5}
			\hline 
			\multirow{2}{*}{$ \minpf(D)  < k $}  
			& no &  \multicolumn{3}{c|}{ $\NEXPTIME$-complete } \\
			\cline{2-5}
			& yes & \multicolumn{2}{c|}{\multirow{1}{*}{$\NPoNP$-complete}}
			& \multicolumn{1}{c|}{\multirow{1}{*}{$\NPoSPoly$-complete}}  \\
			\hline
		\end{tabular}
	}
	\label{table:classe_complexite}
\caption{Complexity results.}
\end{table}

\subsection{Synopsis of proof techniques} 

It is well known that the number of fixed points of a BN is strongly related to the cycles of its SID, see \cite{R19} for a review. Let us define the sign of a cycle as the product of the sign of its arcs. For instance, it is well known that $\maxpf(D)\leq 1$ (resp. $\minpf(D)\geq 1$) if all the cycles of $D$ are negative (resp. positive), and these inequalities are strict if $D$ is strongly connected. The first bound has the following well known generalization: $\maxpf(D)\leq 2^{\tau}$ if we can obtain a SID with only negative cycles by deleting $\tau$ vertices in $D$.  

\medskip
To prove that deciding if $\maxpf(D)\geq 1$ is in $\Poly$, we establish a kind of converse of the first bound. First, we prove that we can compute in polynomial time a SID $D'$, without zero sign, such that $\maxpf(D')\geq 1$ if and only if $\maxpf(D)\geq 1$. Hence, we can assume that $D$ has no zero sign. Under this assumption, we prove that $\maxpf(D)\geq 1$ if and only if each non-trivial initial strongly connected component of $D$ contains a positive cycle. By a deep result in graph theory, one can decide in polynomial time if a SID has a positive cycle, and it follows that deciding if $\maxpf(D)\geq 1$ is in $\Poly$. 

\medskip
All the other results are exact complexity results and, in each case, the main step is the lower bound, which is always done with a reduction from some variations of the SAT problem. The most important case is the reduction used to prove that it is $\NP$-hard to decide if $\maxpf(D)\geq 2$. The reduction is roughly as follows. Given a CNF formula $\psi$ over a set of $n$ variables  we construct in polynomial time a strongly connected SID $H_\psi$ with only positive arcs, with $2n$ particular vertices encoding the positive and negative literals, and with $m$ particular vertices $c_1,\dots,c_m$ corresponding to the clauses, such that every cycle of $D$ contains these $m$ vertices. By the results mentioned above, we have $2\geq \maxpf(H_\psi)\geq\minpf(H_\psi)\geq 2$, and thus $\maxpf(H_\psi)=2$. Let $H'_\psi$ be obtained from $H_\psi$ by adding, for $1\leq s\leq m$, a new vertex $\mu_s$, a positive arcs from each literal contained in the $s$th clause to $\mu_s$, and a negative arc from $\mu_s$ to $c_s$. Then one easily shows that $\maxpf(H'_\psi)< 2$: in some sense, the negative arcs ``destroy'' the positive cycles, and the system behaves as if there are only negative cycles. But now, let $D_\psi$ be obtained from $H'_\psi$ by adding new vertices $\lambda_1,\dots,\lambda_n$ and a positive (resp. negative) arc from $\lambda_r$ to the positive (resp. negative) literal corresponding to the $r$th variable. Then the situation is in-between: we can have $\max(D_\psi)= 2$ or $\max(D_\psi)< 2$ depending on $\psi$ and the initial choice for $H_\psi$. However, by choosing $H_\psi$ carefully, we have that $\max(D_\psi)=2$ if $\psi$ is satisfiable and $\max(D_\psi)<2$ otherwise, and it is thus $\NP$-hard to decide if $\maxpf(D)\geq 2$. 

\medskip
Furthermore, if $\psi$ is a $3$-CNF then $\Delta(D_\psi)\leq 3$, so the complexity lower bound still holds when the maximum in-degree is bounded. With few modifications, we can decrease the maximum in-degree to $2$ without modifying the complexity lower bound. In addition, for a fixed $k>2$, by adding $\lceil \log_2 k \rceil-1$ new vertices in $D_\psi$, and a positive loop on each new vertex, we obtain a SID $D'_\psi$ such that $\maxpf(D'_\psi)\geq k$  if and only if $\maxpf(D_\psi)\geq 2$. Thus, deciding if $\maxpf(D)\geq k$ for each fixed $k\geq 2$ is $\NP$-hard, even if the maximum in-degree of $D$ is bounded. When $k$ is part of the input  and the in-degree is bounded, the reduction is still obtained from minor modifications of $D_\psi$. In that case, the reduction is from \EMSAT, which is $\NPoSPoly$-complete. This problem asks, given a CNF formula $\psi$ over a set of $n$ variables and an integer $1\leq s\leq n$, if there exists an assignment $z$ of the first $s$ variables such that $\psi$ is satisfied by the majority of the $2^{n-s}$ assignments extending $z$.    

\medskip
When $k$ is part of the input and the in-degree is unbounded, the reduction differs from the previous case and is much more technical. This reduction is from \SSAT,  a classical $\NEXPTIME$-complete problem. It asks, given a succinct representation of a $3$-CNF formula $\Psi$, if $\Psi$ is satisfiable. First, we translate the succinct representation under the form of a BN $h$. The formula $\Psi$ is then encoded by the fixed points of $h$: to each clause, and each position in that clause, there is a fixed point encoding the literal contained in the clause in that position. In turn, these fixed points are encoded into a CNF formula $\phi$ (whose variables are the vertices of the SID of $h$). We then construct a SID $D_\Psi$ essentially from the union of the SID of $h$ and the SID $D_{\psi}$ described above, where $\psi$ is obtained from $\phi$ by adding a few extra clauses. It then appears that if the number of fixed points of a BN $f\in F(D_\Psi)$ is sufficiently large, then the fixed points of $f$ extend (in some way) those of $h$ (thanks to the formula $\phi$ contained in $\psi$), and thus encode $\Psi$ and, furthermore, we can extract from $f$ an assignment for $\Psi$ satisfying many clauses. Conversely, if many clauses of $\Psi$ can be simultaneously satisfied, then one can construct a BN $f\in F(D_\Psi)$ with many fixed points. We then obtain that $\Psi$ is satisfiable if and only if $\maxpf(D_\Psi)\geq k$ for some $k$ depending on $D_\Psi$. This shows that, when $k$ is part of the input, it is $\NEXPTIME$-hard to decide if $\maxpf(D)\geq k$. 

\medskip
The duality between positive and negative cycles, highlighted by the basic results given at the beginning of this subsection, allows us to adapt, with few modifications, all above constructions for the study of the minimum number of fixed points. In each reduction, it is enough to take one of the constructions and inverse the sign of a single arc. The analysis is then rather easy, because based on results concerning the maximum number of fixed points. Even if the constructions and the analysis are similar to the maximum number of fixed points, the complexity results differ. In particular, when $k$ is fixed and the maximum in-degree is bounded, the reduction is from \QSAT, which is $\NPoNP$-complete. This problem asks, given a CNF formula $\psi$ over a set of $n$ variables and an integer $1\leq s\leq n$, if there exists an assignment $z$ of the first $s$ variables such that $\psi$ is satisfied by all the assignments extending $z$.

\section{Preliminaries}\label{section:definition}

\subsection{Configurations} 

Given a finite set $V$, a {\em configuration} on $V$ is an element $x\in \bool^V$ that assigns a {\em state} $x_i\in\bool$ to every $i\in V$. For $a\in\bool$, we write $x=a$ to mean that $x_i=a$ for all $i\in V$. Given an enumeration $i_1,i_2,\dots,i_n$ of the elements of $V$, we use the one-line notation writing $x=x_{i_1}x_{i_2}\dots x_{i_n}$. For $I\subseteq V$, we denote by $x_I$ the restriction of $x$ on $I$, that is, the configuration $y\in\bool^I$ such that $y_i=x_i$ for all $i\in I$; and $x$ {\em extends} a configuration $y\in\bool^I$ if $x_I=y$. For every $i\in  V$, we denote the $i$-base vector by $e_i$, that is, $(e_i)_i=1$ and $(e_i)_j=0$ for all $j\neq i$. Given  $x,y\in\bool^V$, we denote by $x \xor y$ the configuration $z$ on $V$ such that $z_i=x_i \xor y_i$ for all $i\in V$, where the addition is computed modulo two. Hence, $x \xor e_i$ is the configuration obtained from $x$ by flipping the state of $i$ only. We write $x\leq y$ to means that $x_i\leq y_i$ for all $i\in V$. 

\subsection{Boolean networks} 

A {\em Boolean network} (BN) with component set $V$ is a function $f : \bool^V \to 
\bool^V$. Given an initial configuration $x^0$ on $V$, the dynamics of the network is described by the successive iterations of $f$, that is, $x^{t+1}=f(x^t)$ for all $t\in\mathbb{N}$. Hence, the state of component $i$ evolves according to the {\em local function} $f_i:\bool^V \to \bool$, which is the coordinate $i$ of $f$, that is, $f_i(x)=f(x)_i$. More generally, given $I\subseteq V$, we denote by $f_I$ the function from $\bool^V$ to $\bool^I$ defined by $f_I(x)=f(x)_I$. If $y\in\bool^I$ then $f_I=y$ means that $f_I$ is a constant function that always returns $y$.

\begin{example}\label{ex:f}
Here is a  BN $f$ with component set $V=\{1,2,3\}$ given under three equivalent forms: a table; a description of the three local functions by logical formulas; and a digraph with an arc from $x$ to $f(x)$ for each configuration $x$ (which is usually called the synchronous or parallel dynamics of $f$). 
\[
\begin{array}{c|c}
x & f(x)\\\hline 
000 & 100 \\
001 & 010 \\
010 & 100 \\
011 & 011 \\
100 & 010 \\
101 & 010 \\
110 & 010 \\
111 & 011 
\end{array}
\qquad
\begin{array}{l}
f_1(x)=\neg x_1\land\neg x_3\\
f_2(x)=x_1\lor x_3\\
f_3(x)=x_2\land x_3
\end{array}
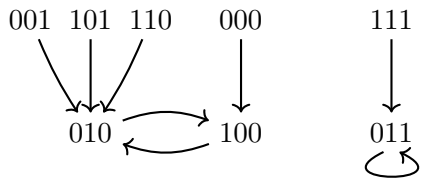
\qquad
\begin{array}{c}
\begin{tikzpicture}
\node (011) at (4,0){$011$};
\node (010) at (0,0){$010$};
\node (100) at (2,0){$100$};
\node (111) at (4,1.5){$111$};
\node (001) at (-0.8,1.5){$001$};
\node (101) at (0,1.5){$101$};
\node (110) at (0.8,1.5){$110$};
\node (000) at (2,1.5){$000$};

\draw[->,thick] (011.-112) .. controls (3,-0.7) and (5,-0.7) .. (011.-68);

\path[->,thick]
(111) edge (011)
(001) edge[bend right=5] (010)
(110) edge[bend left=5] (010)
(101) edge (010)
(000) edge (100)
(010) edge[bend left=20] (100)
(100) edge[bend left=20] (010)
;
\end{tikzpicture}
\end{array}
\]
\end{example}

\subsection{Signed digraphs}

A {\em signed digraph} $D$ consists of a finite set of vertices $V_D$, a set of arcs $A\subseteq V_D\times V_D$ and an arc-labeling function $\sigma$ from $A$ to $\{-1,0,1\}$, which gives a sign (negative, \Null{} or positive) to each arc $(j,i)$, denoted $\sigma_{ji}$. In the following, it will be convenient to set $\tilde\sigma_{ji}=0$ if $\sigma_{ji}\geq 0$ and $\tilde\sigma_{ji}=1$ otherwise. We say that $D$ is {\em \simple{}} if it has no \Null{} arc, and {\em full-positive} if it has only positive arcs. Given a vertex $i$, we denote by $N_D(i)$ the set of in-neighbors of $i$ and, for $s\in\{-1,0,1\}$, we denote by $N^s_D(i)$ the set of in-neighbors $j$ of~$i$ with $\sigma_{ji}=s$. We drop $D$ in the previous notations when it is clear from the context. We call $N^1(i)$, $N^{0}(i)$ and $N^{-1}(i)$ the set of positive, \Null{} and negative in-neighbors of~$i$, respectively. We say that $i$ is a {\em source} if it has no in-neighbor. The maximum in-degree of $D$ is denoted $\Delta(D)$. Cycles and paths are always directed and without repeated vertices. The {\em sign} of a cycle or a path is the product of the signs of its arcs. We say that $i$ has a {\em positive loop} if $D$ has a positive arc from $i$ to itself (a positive cycle of length one). Given $I\subseteq V_D$, we denote by $D\setminus I$ the signed digraph obtained from $D$ by deleting the vertices in $I$ (with the attached arcs). We say that a signed digraph $D'$ is a {\em spanning subgraph} of $D$ if $D'$ can be obtained from $D$ by removing arcs only. We say that $I$ is a  {\em positive feedback vertex set} of $D$ if $D\setminus I$ has only negative cycles (or is acyclic). The minimum size of a positive feedback vertex set of $D$ is denoted $\tau^+(D)$. 

\begin{example}\label{ex:D}
Here is a signed digraph $D$ with $3$ vertices.
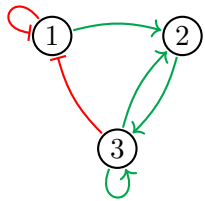
\[
\begin{array}{c}
\begin{tikzpicture}
\node[outer sep=1,inner sep=2,circle,draw,thick] (1) at (150:1){$1$};
\node[outer sep=1,inner sep=2,circle,draw,thick] (2) at (30:1){$2$};
\node[outer sep=1,inner sep=2,circle,draw,thick] (3) at (270:1){$3$};
\draw[red,-|,thick] (1.128) .. controls (140:1.8) and (160:1.8) .. (1.172);
\draw[Green,->,thick] (3.-112) .. controls (-100:1.8) and (-80:1.8) .. (3.-68);
\path[->,thick]
(1) edge[Green,bend left=15] (2)
(2) edge[Green,bend left=15] (3)
(3) edge[red,-|,bend left=15] (1)
(3) edge[Green,bend left=15] (2)
;
\end{tikzpicture}
\end{array}
\]
Green arrows represent positive arcs and T-end red arrows represent negative arcs; this convention is used throughout the paper. Hence, $D$ is nice. It has two positive cycles (of lengths $1$ and $2$) and two negative cycles (of lengths $1$ and $3$). Note that $\{3\}$ is a positive feedback vertex set, hence $\tau^+(D)=1$. 
\end{example}

\subsection{Signed interaction digraphs}

The {\em signed interaction digraph} (SID) of a BN $f$ with component set $V$ is the signed digraph $D_f$ with vertex set $V_D=V$ defined as follows. First, given $i,j \in V$, there is an arc $(j,i)$ if and only if there exists a configuration $x$ on $V$  such that $f_i(x \xor e_j) \neq f_i(x)$ ({\ie} the local function $f_i$ depends on component $j$). Second, the sign $\sigma_{ji}$ of an arc $(j,i)$ depends on whether the local function $f_i$ is increasing or decreasing with respect to component $j$, and is defined as
\[
\sigma_{ji}=\left\{\begin{array}{rl}
1  &\quad \text{if  $f_i(x \xor e_j)\geq f_i(x)$ for all $x\in\bool^V$ with $x_j=0$,}\\
-1 &\quad \text{if  $f_i(x \xor e_j)\leq f_i(x)$ for all $x\in\bool^V$ with $x_j=0$,}\\
0  &\quad \text{otherwise.}
\end{array}\right.
\]
Hence, a positive (resp. negative, \Null{}) arc from $j$ to $i$ means that $f_i(x)$ is an increasing (resp. decreasing, non-monotone) function of $x_j$.

\begin{example}
The SID of the BN $f$ given in Example~\ref{ex:f} is the signed digraph $D$ given in Example~\ref{ex:D}. One can easily deduce that from description of the three local functions by logical formulas. First, we have $f_1(x)=\neg x_1\land \neg x_3$, thus $f_1(x)$ only depends on $x_1$, $x_3$, and is decreasing with $x_1,x_2$. Hence, in the SID of $f$, vertex $1$ has two in-coming arcs, from vertex $1$ and vertex $3$, both negative. Second, we have $f_2(x)= x_1\lor x_3$, thus $f_2(x)$ only depends on $x_1$, $x_3$, and is increasing with $x_1,x_3$. Hence, in the SID of $f$, vertex $2$ has two in-coming arcs, from vertex $1$ and vertex $3$, both positive. Third, we have $f_3(x)= x_2\land x_3$, thus $f_3(x)$ only depends on $x_2$, $x_3$, and is increasing with $x_2,x_3$. Hence, in the SID of $f$, vertex $3$ has two in-coming arcs, from vertex $2$ and vertex $3$, both positive.
\end{example}

A signed digraph is called \an{} SID if it is the SID of at least one BN. Such signed digraphs are characterized below (as a consequence, every \simple{} signed digraph is \an{} SID). 

\begin{proposition}[\cite{PR10}]\label{pro:SID}
A signed digraph $D$ is \an{} SID if and only if $|N^0_D(i)|\neq 1$ or $|N_D(i)|\geq 3$ for all vertex $i$.
\end{proposition}

A fundamental remark regarding the present work is that multiple BNs may have the same SID. Given \an{} SID $D$, we denote by $\functions(D)$ the corresponding BNs:
\[
\functions(D)=\{f:\bool^{V_D}\to\bool^{V_D}\mid D_f=D\}.
\]
The size of $\functions(D)$ is generally large: there are $2^{n2^n}$ BNs with $n$ components and only $3^n$ signed digraphs with $n$ vertices. Given $i\in V_D$, we denote by $\functions_i(D)$ the possible local functions for $i$, that is, 
\[
\functions_i(D)=\{f_i\mid f\in \functions(D)\}.
\]
The size of $\functions_i(D)$ is doubly exponential according to the in-degree $d$ of $i$ in $D$, thus it scales as the number of Boolean functions on $d$ variables, $2^{2^d}$. The precise value of $|\functions_i(D)|$ is not trivial and has been extensively studied when $i$ has only positive in-neighbors, see {\tt A006126} on the OEIS~\cite{oeisA00616}. 

\medskip
In the following, we will often consider vertices with in-degree $\leq 2$ and, in that case, the situation is clear. If $i$ is a source, then there are only two possible local functions, the two constant local functions. If $i$ has a unique in-neighbor, say $j$, then it is not a \Null{} in-neighbor by Proposition \ref{pro:SID}, and there is a unique possible local function: we have $f_i(x)=x_j\xor\tilde \sigma_{ji}$. If $i$ has only positive or negative in-neighbors, we say that $f_i$ is the AND (resp. OR) function if it is the ordinary logical and (resp. or) but inputs with a negative sign are flipped, that~is, 
\[
  f_i(x)=\bigwedge_{j\in N(i)} x_j \xor \tilde\sigma_{ji}\qquad\text{(resp.}~ 
f_i(x)=\bigvee_{j\in N(i)} x_j \xor \tilde\sigma_{ji}
~\text{)}.
\] 
(If $i$ is of in-degree one, the AND and OR functions are identical.) Now, if $i$ is of in-degree two and has no \Null{} in-neighbor, there are only two possibles local functions: the AND function and the OR function.   

\begin{example}\label{ex:F(D)}
Take again the signed digraph $D$ of Example~\ref{ex:D}. It is nice, and each vertex is of in-degree two. So for each $f\in\functions(D)$ and $i\in\{1,2,3\}$ the local functions $f_i$ is either the AND function of the OR function. Thus, $|\functions_{i}(D)|=2$ for $i\in\{1,2,3\}$ and $|\functions(D)|=8$. The table of the $8$ BNs in $\functions(D)$, denoted from $f^1$ to $f^8$, is given below; fixed points are in bold. Note that the BN considered in Example~\ref{ex:f} is $f^3$. 
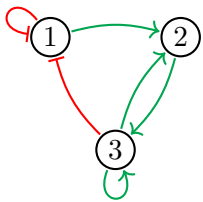
\[
\begin{array}{c}
\begin{tikzpicture}
\node[outer sep=1,inner sep=2,circle,draw,thick] (1) at (150:1){$1$};
\node[outer sep=1,inner sep=2,circle,draw,thick] (2) at (30:1){$2$};
\node[outer sep=1,inner sep=2,circle,draw,thick] (3) at (270:1){$3$};
\draw[red,-|,thick] (1.128) .. controls (140:1.8) and (160:1.8) .. (1.172);
\draw[Green,->,thick] (3.-112) .. controls (-100:1.8) and (-80:1.8) .. (3.-68);
\path[->,thick]
(1) edge[Green,bend left=15] (2)
(2) edge[Green,bend left=15] (3)
(3) edge[red,-|,bend left=15] (1)
(3) edge[Green,bend left=15] (2)
;
\end{tikzpicture}
\end{array}
\qquad
\begin{array}{c|c|c|c|c|c|c|c|c}
x & f^{1}(x) & f^{2}(x) & f^{3}(x) & f^{4}(x) & f^{5}(x) & f^{6}(x) & f^{7}(x) & f^{8}(x) \\\hline
000 & 100 & 100 & 100 & 100 & 100 & 100 & 100 & 100 \\
001 & 000 & \bold{001} & 010 & 011 & 100 & 101 & 110 & 111 \\
010 & 100 & 101 & 100 & 101 & 100 & 101 & 100 & 101 \\
011 & 001 & 001 & \bold{011} & \bold{011} & 101 & 101 & 111 & 111 \\
100 & 000 & 000 & 010 & 010 & \bold{100} & \bold{100} & 110 & 110 \\
101 & 010 & 011 & 010 & 011 & 010 & 011 & 010 & 011 \\
110 & 000 & 001 & 010 & 011 & 100 & 101 & \bold{110} & 111 \\
111 & 011 & 011 & 011 & 011 & 011 & 011 & 011 & 011 \\
\end{array}
\]
\end{example}

\subsection{Fixed points and decision problems}
 
A {\em fixed point} of $f$ is a configuration $x$ such that $f(x)=x$. We denote by $\phi(f)$ the number of fixed points of $f$. In this paper, we are interested in decision problems related to the maximum and minimum number of fixed points of BNs in $\functions(D)$, denoted
\[
\begin{array}{l}
\maxpf(D)=\max\left\{\phi(f) \mid f \in \functions(D)\right\},\\[2mm]
\minpf(D)=\min\left\{\phi(f) \mid f \in \functions(D)\right\}.
\end{array}
\]

\begin{example}
For the signed digraph $D$ of Example~\ref{ex:D}, we have $\minpf(D)=0$ and $\maxpf(D)=1$. This can be easily deduce from the table given in Example~\ref{ex:F(D)}.
\end{example}

More precisely, we will study the complexity of deciding if $\maxpf(D)\geq k$ or $\minpf(D) < k$, where $k$ is a positive integer, fixed or not. This gives the following decision problems. 

\begin{quote}
	\decisionpb
	{\KMaxFPP ~(\kmaxfpp{k})}
	{\an{} SID $D$.}{$\maxpf(D) \geq  k$?}
\end{quote}

\begin{quote}
	\decisionpb
	{\MaxFPP ~(\maxfpp)}
	{\an{} SID $D$ and an integer $k\geq 1$.}{$\maxpf(D) \geq  k$?}
\end{quote}

\begin{quote}
	\decisionpb
	{\KMinFPP ~(\kminfpp{k})}
	{\an{} SID $D$.}{$\minpf(D) <  k$?}
\end{quote}

\begin{quote}
	\decisionpb
	{\MinFPP~(\minfpp)}
	{\an{} SID $D$ and an integer $k\geq 1$.}{$\minpf(D) <  k$?}
\end{quote}

\medskip
All these problems are in $\NEXPTIME$ (they can be decided in exponential time on a non-deterministic Turing machine). For instance the problem \maxfpp{} can be decided as follows. Given \an{} SID $D$ with $n$ vertices and an integer $k$:
\begin{enumerate}
	\item 
	Choose non-deterministically a BN $f$ with component set $V_D$; this can be done in exponential time since $f$ can be represented using $n2^{n}$ bits. 
	\item 
	Compute the SID $D_f$ of $f$; this can be done in exponential time by comparing $f_i(x)$ and $f_i(x\xor e_j)$ for all configurations $x$ on $V_D$ and vertices $i,j\in V_D$.  
	\item 
	Compute $\phi(f)$ by considering each of the $2^n$ configurations.
	\item 
	Accept if and only if $\phi(f) \geq k$ and $D_f=D$.
\end{enumerate}
This non-deterministic exponential time algorithm has an accepting branch if and only if $\maxpf(D) \geq  k$. Therefore, the problem \maxfpp{} is in $\NEXPTIME$. This algorithm can be adapted to the other problems.

\medskip
However, as we will see later, this complexity bound can be refined for some problems or when we restrict the problems to some subclass of SIDs, such as SIDs with a maximum in-degree bounded by a constant $d$. This is because the SIDs of this subclass only admit a simple exponential number of BNs. Indeed, if $f \in \functions(D)$ and $\Delta(D)\leq d$, then each local function $f_i$ can be regarded as a Boolean function with $|N(i)|\leq d$ inputs. Consequently, there are at most $c = 2^{2^d}$ possible choices for each local function, and $|\functions(D)| \leq c^n$ is a simple exponential. 

\medskip
Remark~\ref{remark:delta_1} below shows that the restriction to SIDs $D$ with maximum in-degree $\leq 1$ reduces drastically the hardness of the problems. Therefore, in the rest of the article, we will only consider SIDs whose maximum in-degree is bounded by a constant $d \geq 2$ or not bounded at all.

\begin{remark} \label{remark:delta_1}
The decisions problems \kmaxfpp{k}, \maxfpp{}, \kminfpp{k} and \minfpp{} restricted to SIDs $D$ such that $\Delta(D) \leq 1$ are in $\Poly$. Indeed, if $\Delta(D) \leq 1$ then the numbers $c^+$ and $c^-$ of positive and negative cycles can be computed in polynomial time (and $c^-+c^+$ is at most the number of vertices). Furthermore, there is a unique $f\in\functions(D)$, and one easily check that 
	\[
	\phi(f) =
	\begin{cases}
	0 & \text{if $c^->0$}, \\
	2^{c^+} & \text{otherwise}.
	\end{cases}
	\]
As a consequence, computing $\phi(f)$, which is obviously equal to both $\maxpf(D)$ and $\minpf(D)$, can be done in polynomial time, and it is sufficient to compare this value to $k$ in order to answer any problem.
\end{remark}

\subsection{Basic results}

Cycles of $D$ are known to play a fundamental role concerning the number of fixed points (and the above remark already shows this). A basic result by Robert is that if $D$ is acyclic, then every $f\in \functions(D)$ has a unique fixed point (thus $\maxpf(D)=\minpf(D)=1$). Considering positive and negative cycles, Aracena \cite{A04} proved the following other basic results, that show a kind of duality between the two types of cycle (note that Robert's result is an immediate consequence of the first two items).

\begin{theorem}[\cite{A04}]\label{thm:basic_results}
Let $D$ be \an{} SID.
\begin{enumerate}
\item If $D$ has only negative cycles then $\maxpf(D)\leq 1$.
\item If $D$ has only positive cycles then $\minpf(D)\geq 1$.
\item If $D$ is strongly connected and has only negative cycles then $\maxpf(D)<1$.
\item If $D$ is strongly connected and has only positive cycles then $\minpf(D)>1$. 
\end{enumerate}
\end{theorem}

\medskip
The first item can be widely generalized into the following bound (many extensions or improvements for particular classes of SIDs  exist; see \cite{R08,GRR15,ARS17} for instance).  

\begin{theorem}[Positive feedback bound \cite{A04}]\label{thm:bound}
For every SID $D$ we have $\maxpf(D)\leq 2^{\tau^+(D)}$.
\end{theorem}

The positive feedback bound is an immediate consequence of the following lemma, used many times in the following.

\begin{lemma}[\cite{A04}]\label{lem:positive_feedback_lemma}
Let $D$ be \an{} SID, let $I$ be a positive feedback vertex set of $D$ and $f\in\functions(D)$. If $f$ has distinct fixed points $x,y$, then $x_I\neq y_I$. 
\end{lemma}

The proof of all the previous results uses at some point the following lemma, that we will also use intensively (the first item appears in \cite{GRR15} and the second is a consequence of the first; we give a proof for completeness). Given \an{} SID $D$ and $i\in V_D$, we define the partial order $\leq^D_i$ on $\bool^{V_D}$ by 
\[
x\leq^D_i y
\iff 
x_{N^1(i)}\leq y_{N^{1}(i)}\text{ and }
y_{N^{-1}(i)}\leq x_{N^{-1}(i)}\text{ and }
x_{N^0(i)}= y_{N^{0}(i)}.
\]

\begin{lemma} \label{lem:partial_order}
Let $D$ be a SID with sign function $\sigma$, and let $f\in \functions(D)$. For every $i\in V_D$, we have the following properties:
\begin{enumerate}
\item \label{prop1} 
The local function $f_i$ is non-decreasing with respect to $\leq^D_i$, that is, for any configurations $x,y$ on $V_D$ we have 
	\[
	x\leq^D_i y\quad\Rightarrow\quad f_i(x)\leq f_i(y).
	\]
\item\label{prop2}
If $i$ has at least one in-neighbor and at most one \Null{} in-neighbor then, for any configuration $x$ on $V_D$, there is a non-\Null{} in-neighbor $j$ of $i$ such that $f_i(x)=x_j\xor\tilde\sigma_{ji}$.
\end{enumerate}
\end{lemma}

\begin{proof}
To prove the first property, let $x,y$ be configurations on $V_D$ such that $x\leq^D_i y$. Let $\Delta(x,y)$ be the set of $j\in V_D$ such that $x_j\neq y_j$. We proceed by induction on $|\Delta(x,y)|$, the Hamming distance between $x$ and $y$. If $|\Delta(x,y)|=0$ then $x=y$ so $f_i(x)=f_i(y)$. Otherwise, there exists $j\in \Delta(x,y)$ and we set $z=x\xor e_j$. Since $z_j=y_j$, it is clear that $z\leq^D_i y$. Furthermore, we have $\Delta(z,y)=\Delta(x,y)\setminus \{j\}$ so $f_i(z)\leq f_i(y)$ by induction hypothesis. If $j\not\in N(i)$,  we have  $f_i(x)=f_i(z)\leq f_i(y)$. Otherwise, $j$ is either a positive or a negative in-neighbor of~$i$ (since $x_j\neq y_j$ and $x\leq^D_i y$). If $j$ is a positive in-neighbor, then $x_j<y_j=z_j$ so $f_i(x)\leq f_i(z)$, and we deduce that $f_i(x)\leq f_i(y)$. If $j$ is a negative in-neighbor, then $x_j>y_j=z_j$, so $f_i(x)\leq f_i(z)$ and we have again $f_i(x)\leq f_i(y)$. This concludes the induction step.
	
\medskip
We now prove the second property. Suppose that $f_i(x)=0$ and suppose,
for a contradiction, that there is no non-\Null{} in-neighbor $j$ of $i$
such that $x_j\xor\tilde\sigma_{ji}=0$. This means that
$x_{N^{1}(i)}=1$ and $x_{N^{-1}(i)}=0$ or, equivalently, that $x$ is a
$\leq^D_i$-maximal configuration. If $i$ has no \Null{} in-neighbor, for
any configuration $y$ we have $y\leq^D_i x$ and we deduce from the
first property that $f_i(y)=0$. Hence, $f_i$ is the 0 constant
function, which is a contradiction since $i$ has at least one
in-neighbor. Suppose now that $i$ has a unique \Null{} in-neighbor, say
$k$. For any configuration $y$ with $y_k=x_k$ we have $y\leq^D_i x$
and we deduce from the first property that $f_i(y)=0$, so that
$f_i(y)\leq f_i(y\xor e_k)$. It follows that $k$ is a positive
in-neighbor if $x_k=0$ and a negative in-neighbor if $x_k=1$, a
contradiction. The case $f_i(x)=1$ is similar.
\end{proof}

\section{\boldmath\KMaxFPP{} for $k = 1$\unboldmath} \label{section:max_k=1}

In this Section, we are interested in the problem of deciding if \an{} SID $D$ admits a BN with at least one fixed point. This is the only decision problem we consider for which we do not prove tight complexity bounds. More precisely, Theorem~\ref{theorem:kmaxfpp1} shows that the problem is in $\Poly$, but it remains open to know whether it is $\Poly$-hard.

\begin{theorem} \label{theorem:kmaxfpp1}
\kmaxfpp{1} is in $\Poly$. 
\end{theorem}

The next lemma shows that we can efficiently transform any SID $D$ into a \simple{} SID $D'$ such that $ \maxpf (D) \geq 1$ if and only if $ \maxpf (D') \geq 1$. As a consequence, in order to prove that \kmaxfpp{1} is in $\Poly$, we can consider that the input SID is \simple{}.

\begin{lemma}\label{lemma:simple_k1}
Let $D$ be \an{} SID and let $D'$ be the \simple{} SID obtained from $D$ by deleting the set of arcs $(j,i)$ such that $i$ has at least two \Null{} in-neighbors in $D$ or such that $j$ is the unique \Null{} in-neighbor of $i$ in $D$. We have 
\[
\maxpf(D) \geq 1
\iff 
\maxpf(D') \geq 1.
\]
\end{lemma}

\begin{proof}
Suppose that $\maxpf(D)\geq 1$ and let $f \in \functions(D)$ with a fixed point $y$. We define componentwise  a BN $f'\in \functions(D')$ that admits $y$ as fixed point. Let $i$ be any vertex. We consider three cases:
\begin{enumerate}
\item
If $i$ has no \Null{} in-neighbor in $D$, then $i$ has the same in-coming arcs in $D$ and $D'$ thus we can set $f'_i=f_i$, so that $f'_i(y)=f_i(y)=y_i$. 
\item
If $i$ has a unique \Null{} in-neighbor in $D$, then $f'_i$ is defined as the AND
function if $y_i=0$, and the OR function otherwise. Suppose first that
$y_i=0$. Then $f_i(y)=0$ so, by the second property of
Lemma~\ref{lem:partial_order}, there is a non-\Null{} in-neighbor $j$ of $i$
such that $y_j\xor\tilde\sigma_{ji}=0$, where $\sigma$ is the sign function
of $D$. Since $f'_i$ is the AND function, we deduce that $f'_i(y)=0$. We
prove similarly that $f'_i(y)=1$ when $y_i=1$.  
\item
If $i$ has at least two \Null{} in-neighbors in $D$, then $i$ is a source of $D'$ so $f'_i$ has to be a constant function. We then set $f'_i=y_i$, so that $f'_i(y)=y_i$. 
\end{enumerate}
In this way, $f'$ is a BN in $\functions(D')$ which has $y$ as fixed point. Hence, $\maxpf(D')\geq 1$.

\medskip
Conversely, suppose that $\maxpf(D')\geq 1$ and let  $f' \in \functions(D')$ with a fixed point $y$. We define componentwise a BN $f\in \functions(D)$ that admits $y$ as fixed point. Let $i$ be any vertex. We consider three cases:
\begin{enumerate}
\item
If $i$ has no \Null{} in-neighbor in $D$, then $i$ has the same in-coming arcs in $D$ and $D'$ thus we can set $f_i=f'_i$, so that $f_i(y)=f'_i(y)=y_i$.
\item
Suppose that $i$ has a unique \Null{} in-neighbor in $D$, say $k$. Note that $i$
has at least two non-\Null{} in-neighbors in $D$ by Proposition \ref{pro:SID}.
Suppose that $y_i=0$. Then $f'_i(y)=0$ so, by the second property of
Lemma~\ref{lem:partial_order}, there is a non-\Null{} in-neighbor $j$ of $i$
such that $y_j\xor\tilde\sigma_{ji}=0$. We then define $f_i$ by:
\[
f_i(x)=\bigl((x_j \xor \tilde{\sigma}_{ji}\big) \vee (x_k \xor y_k) \bigr) 
\wedge \bigwedge\limits_{ \ell \in N_{D}(i) \setminus \{j,k\} }  \bigl((x_\ell \xor \tilde{\sigma}_{\ell i}) \vee (x_k \xor \neg y_k) \bigr).
\]
It is easy to check that $f_i \in \functions_i(D)$ and, since the first term of the conjunction vanishes for $x=y$, we have $f_i(y)=0=y_i$. If $y_i=1$ we prove similarly that there is $f_i\in \functions_i(D)$ with $f_i(y)=1$ (in that case, there is a non-\Null{} in-neighbor $j$ such that $y_j\xor \tilde\sigma_{ji}=1$ and $f_i$ is defined as above by swapping $\wedge$ and $\vee$, and by swapping $y_k$ and $\neg y_k$).
\item
Suppose that $i$ has at least two \Null{} in-neighbors in $D$. If $y_i=0$ we define $f_i$ by:
\[
f_i(x) =
\left(\bigoplus\limits_{j \in N^0_D(i)} x_j\xor y_j \right) \wedge
\bigwedge\limits_{j \in N_D(i) \setminus N^0_D(i)} (x_j \xor \tilde\sigma_{ji}).
\] 
It is easy to check that $f_i \in \functions_i(D)$ and, since the first term of
the conjunction vanishes for $x=y$, we have $f_i(y)=0=y_i$. The case
$y_i=1$ is symmetric (with $\vee$ instead of $\wedge$ and $\neg y_j$
instead of $y_j$). 
\end{enumerate}
In this way, $f$ is a BN in $\functions(D)$ which has $y$ as fixed point. Hence, $\maxpf(D)\geq 1$.
\end{proof}


We now give a graph-theoretical characterization of the \simple{} SIDs $D$ such that $\maxpf(D) \geq 1$. We need some additional definitions. A strongly connected component $H$ in a signed digraph $D$ is {\em trivial} if it has a unique vertex and no arc, and {\em initial} if $D$ has no arc $(i,j)$ where $j$ is in $H$ but not $i$. 

\begin{lemma}\label{lemma:kmaxfpp1_eq_evencycle}
Let $D$ be a \simple{} SID. We have $\maxpf(D) \geq 1$ if and only if each non-trivial initial strongly connected component of $D$ has a positive cycle. 
\end{lemma}

\begin{proof}
The left to right implication has been proven by Aracena
in~\cite[Corollary~3]{A08} and is an easy consequence of the third item in
Theorem~\ref{thm:basic_results}. We present a version rewritten with the
notations of this paper. Let $\sigma$ be the sign function of $D$, and let $f
\in \functions(D)$  with a fixed point $x$. Consider an arbitrary non-trivial
initial strongly connected component $H$ of $D$, and let us prove that $H$
has a positive cycle. Since $H$ is non-trivial, by the second property of
Lemma~\ref{lem:partial_order}, each vertex $i$ in $H$ has a non-\Null{}
in-neighbor $j$ such that $x_j\xor\tilde\sigma_{ji}=f_i(x)=x_i$,  and $j$ is
necessarily in $H$ since $H$ is initial. We deduce that $H$ has a spanning
subgraph $H'$ in which each vertex is of in-degree one, and such that
$x_j\xor\tilde\sigma_{ji}=x_i$ for any arc $(j,i)$ of $H'$. Hence, for any
vertices $i,j$ in $H$, if $x_j=x_i$ then any walk in $H'$ from $j$ to $i$
visits an even number of negative arcs. In particular, any cycle of $H'$ is
positive. Since $H'$ has no source, $H'$ has a cycle, which is positive.
Thus, $H$ has a positive cycle as desired.

\medskip
Conversely, suppose that each non-trivial initial strongly connected component
of $D$ has a positive cycle. Then it is easy to see that $D$ has a spanning
subgraph $D'$ with only positive cycles, and with the same sources as $D$. Let
any $f'\in \functions(D')$. By the second item of Theorem
\ref{thm:basic_results}, $f'$ has at least one fixed point, say $x$. We then
define $f\in \functions(D)$ componentwise as follows. For any vertex $i$, if
$i$ is a source, then we set $f_i=x_i$, so that $f_i(x)=x_i$. Otherwise, by the
second property of Lemma~\ref{lem:partial_order}, $i$ has an in-neighbor $j$ in
$D'$ such that $x_j\xor\tilde\sigma_{ji}=f'_i(x)=x_i$. We then define $f_i$ as
the AND function if $x_j\xor\tilde\sigma_{ji}=x_i=0$ and the OR function if
$x_j\xor\tilde\sigma_{ji}=x_i=1$. It is then clear that $f_i(x)=x_i$. Hence,
$x$ is a fixed point of $f$.\end{proof}

\begin{remark}
If $D$ is not \simple{}, we can have $\maxpf(D)=0$ even if each non-trivial initial strongly connected component of $D$ has a positive cycle. Indeed, let $D$ and $D'$ be the following SIDs:
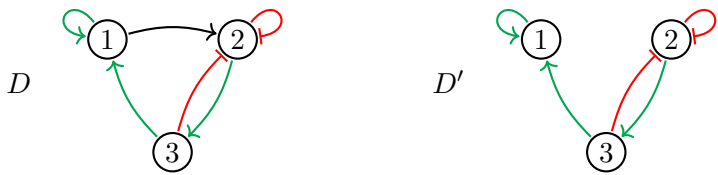
\[
D
\begin{array}{c}
\begin{tikzpicture}
\node[outer sep=1,inner sep=2,circle,draw,thick] (1) at (150:1){$1$};
\node[outer sep=1,inner sep=2,circle,draw,thick] (2) at (30:1){$2$};
\node[outer sep=1,inner sep=2,circle,draw,thick] (3) at (270:1){$3$};
\draw[Green,->,thick] (1.128) .. controls (140:1.8) and (160:1.8) .. (1.172);
\draw[red,-|,thick] (2.52) .. controls (40:1.8) and (20:1.8) .. (2.8);
\path[->,thick]
(1) edge[bend left=15] (2)
(2) edge[Green,bend left=15] (3)
(3) edge[Green,bend left=15] (1)
(3) edge[red,-|,bend left=15] (2)
;
\end{tikzpicture}
\end{array}
\qquad\qquad
D'
\begin{array}{c}
\begin{tikzpicture}
\node[outer sep=1,inner sep=2,circle,draw,thick] (1) at (150:1){$1$};
\node[outer sep=1,inner sep=2,circle,draw,thick] (2) at (30:1){$2$};
\node[outer sep=1,inner sep=2,circle,draw,thick] (3) at (270:1){$3$};
\draw[Green,->,thick] (1.128) .. controls (140:1.8) and (160:1.8) .. (1.172);
\draw[red,-|,thick] (2.52) .. controls (40:1.8) and (20:1.8) .. (2.8);
\path[->,thick]
(2) edge[Green,bend left=15] (3)
(3) edge[Green,bend left=15] (1)
(3) edge[red,-|,bend left=15] (2)
;
\end{tikzpicture}
\end{array}
\]
The black arrow from vertex $1$ to vertex $2$ represents a \Null{} arc, thus $D$
is not \simple{}. Furthermore, $D$ is strongly connected and has a positive
cycle: the positive loop on vertex $1$. However, $\maxpf(D)=0$. Indeed, by
Lemma~\ref{lemma:simple_k1}, $\maxpf(D)\geq 1$ if and only if $\maxpf(D')\geq
1$. However, since $D'$ is \simple{} and has a non-trivial initial strongly
connected component with only negative cycles, that containing vertices 2 and
$3$, by Lemma~\ref{lemma:kmaxfpp1_eq_evencycle} we have $\maxpf(D')=0$ and
thus $\maxpf(D)=0$. 
\end{remark}

The last ingredient for the proof of Theorem~\ref{theorem:kmaxfpp1} is a difficult result independently proven by Robertson, Seymour and
Thomas in \cite{RST99}, and McCuaig in \cite{McC04}. 

\begin{theorem} [\cite{McC04,RST99}] \label{th:McCuaig}
There is a polynomial time algorithm to decide if a given digraph contains a cycle of even length.
\end{theorem}

\begin{proof}[Proof of Theorem~\ref{theorem:kmaxfpp1}]
As a consequence of Lemmas~\ref{lemma:simple_k1}
and~\ref{lemma:kmaxfpp1_eq_evencycle}, to decide if $\maxpf(D)\geq 1$, it is
sufficient to compute the non-trivial initial strongly connected components
of the \simple{} SID $D$ (this can be done in linear time \cite{T72}) and to
check if each of them contains a positive cycle. Using the following
transformation, the algorithm from Theorem~\ref{th:McCuaig} can be used to
perform this in polynomial time.
  
\medskip
Let $D$ be a signed digraph with $n$ vertices, and let $\tilde D$ be the digraph obtained from $D$ by replacing each positive arc by a path of length two (with one new vertex), and each negative arc by a path of length one. Then $\tilde D$ has at most $n+n^2$ vertices, and $D$ has a positive cycle if and only if $\tilde D$ has a cycle of even length (this transformation also appears in \cite{MAG08}). This concludes the proof of Theorem~\ref{theorem:kmaxfpp1}.
\end{proof}

\section{\boldmath\KMaxFPP{} for $k \geq 2$\unboldmath} \label{section:max_k>=2}

In this section we study \kmaxfpp{k} for $k \geq 2$ and prove the following complexity result. 

\begin{theorem}
	\label{thm:NP-complete} 
\kmaxfpp{k} is $\NP$-complete for every $k\geq 2$, even when restricted to SIDs $D$ such that $\Delta(D)\leq 2$.
\end{theorem}

We first prove the upper bound and then the lower bound.

\subsection{Upper bound} 

\begin{lemma} \label{lem:NP}
\kmaxfpp{k} is in $\NP$ for any $k\geq 2$.
\end{lemma}

\begin{proof}
Let $D$ be \an{} SID. Suppose that $k \leq 2^{|V_D|}$, otherwise
$D$ is obviously a false instance. The algorithm we consider is fairly
simple. It first guesses $k$ fixed points and, for each non-negative (resp.
non-positive) arc $(j,i)$ in $D$, it guesses a configuration in which an
increase of component $j$ produces an increase (rep. decrease) of the local
function $f_i$. It finally checks that these non-deterministic guesses do not
contradict the first property of Lemma~\ref{lem:partial_order}. If so, this
partial knowledge of the local functions can be extended into a BN in
$\functions(D)$ with the $k$ guessed fixed points, \ie{} $D$ is a true
instance. More precisely, the algorithm is as follows:
\begin{enumerate}
\item 
Choose non-deterministically $k$ distinct configurations $x^1,\dots,x^k$ on $V_D$ and, for every $i\in V_D$, compute the sets $F_i=\{x^\ell\,|\,x^\ell_i=0,\ell\in [k]\}$ and $T_i=\{x^\ell\,|\,x^\ell_i=1,\ell\in [k]\}$.
\item 
For each non-negative arc $(j,i)$ of $D$, choose non-deterministically a configuration $x^{ji+}$ on $V_D$ with $x^{ji+}_j=0$. Then add $x^{ji+}$ in the set $F_i$ and $x^{ji+}\xor e_j$ in the set $T_i$. 
\item 
For each non-positive arc $(j,i)$ of $D$, choose non-deterministically a configuration $x^{ji-}$ on $V_D$ with $x^{ji-}_j=0$. Then add $x^{ji-}$ in the set $T_i$ and $x^{ji-}\xor e_j$ in the set $F_i$.
\item 
Accept if and only if there is no $i\in V_D$, $x\in F_i$ and $y\in T_i$ such that $y\leq^D_i x$.  
\end{enumerate}

\medskip
This algorithm runs in $O(|V_D|\Delta(D)^3)$, which is actually the time complexity of the last item, the most consuming one. Indeed, we have $|F_i||T_i|\leq (\Delta(D)+k)^2$, and the relation $y\leq^D_i x$ can be checked in $O(\Delta(D))$. 

\medskip
We will prove that there is an accepting branch if and only if $\maxpf(D)\geq k$.  

\medskip
Suppose first that there is $f \in \functions(D)$ with at least $k$ fixed points.
One can consider the following execution. First, the configurations
$x^1,\dots,x^k$ chosen in the first step are fixed points of $f$. Second, for
each non-negative arc $(j,i)$ of $D$, the configuration $x^{ji+}$ chosen in
the second step is such that $f_i(x^{ji+}) < f_i(x^{ji+}\xor e_j)$; this
configuration exists since $D$ is the SID of $f$. Third, for each
non-positive arc $(j,i)$, the configuration $x^{ji-}$ chosen in the third
step is such that $f_i(x^{ji-}) > f_i(x^{ji-}\xor e_j)$; this configuration
exists since $D$ is the SID of $f$. Hence, for all $i\in V_D$, $x\in F_i$ and
$y\in T_i$ we have $f_i(x)=0$ and $f_i(y)=1$ and we deduce from the first
property of Lemma~\ref{lem:partial_order} that $y\not\leq^D_i x$, so the
algorithm as an accepting branch.

\medskip
For the other direction, suppose that the algorithm has an accepting branch. Let us prove that there is $f \in \functions(D)$ with at least $k$ fixed points. For each $i \in V_D$, we define $f_i$ has follows: for all configurations $x$ on $V_D$,
	\[
	f_i(x)=\left\{
	\begin{array}{ll}
	1&\text{if there is $y\in T_i$ such that $y\leq^D_i x$},\\
	0&\text{otherwise.}
	\end{array}
	\right.
	\]
Clearly, $f_i(x)=1$ if $x\in T_i$. Furthermore, if $x\in F_i$ then, since we are considering an accepting branch, there is no $y\in T_i$ such that $y\leq^D_i x$, and we deduce that $f_i(x)=0$. Consequently, $f_i(x^\ell)=x^\ell_i$ for $1\leq \ell\leq k$. Thus, $x^1,\dots,x^k$ are fixed points of $f$. 

\medskip
It remains to prove that the SID $D_f$ of $f$ is equal to $D$. Suppose first that $D_f$ has a non-negative arc $(j,i)$. Then there is a configuration $x$ with $x_j=0$ such that $f_i(x)<f_i(x\xor e_j)$. We deduce that $x \not\leq^D_i y \leq^D_i x\xor e_j$ for some $y\in T_i$. So $(j,i)$ is an arc of $D$ and, since $x_j\leq y_j$ and $y \leq^D_i x\xor e_j$, if this arc is negative then $y\leq^D_i x$, a contradiction. Thus, $(j,i)$ is a non-negative arc of $D$. We prove similarly that every non-positive arc of $D_f$ is a non-positive arc of $D$. Now, suppose that $D$ has a non-negative arc $(j,i)$ and let $x=x^{ji+}$ be the corresponding configuration chosen in the second step. We have $x\in F_i$ and  $x\xor e_j\in T_i$, so $f_i(x)<f_i(x\xor e_j)$ and since $x_j=0$ we deduce that $(j,i)$ is a non-negative arc of $D'$. We prove similarly that every non-positive arc of $D$ is a non-positive arc of $D_f$. This proves that $D_f=D$. 
\end{proof}

\subsection{Lower bound}

We now prove the lower bound. We first show that the problem can be reduced to $k=2$: by the following lemma, \kmaxfpp{2} is as hard as \kmaxfpp{k} for all $k>2$. 

\begin{lemma}\label{lem:k_to_2}
Let $k>2$ and let $D$ be any SID. Let $D'$ be the SID obtained from $D$ by adding $\lceil \log_2 k \rceil-1$ new vertices and a positive loop on each new vertex. We have 
\[
\maxpf(D')\geq k \iff \maxpf(D)\geq 2.
\]
\end{lemma} 

\begin{proof}
Let $\ell=\lceil \log_2 k \rceil$ so that $2^{\ell-1}<k\leq 2^\ell$ and $\ell\geq 2$. Let $H$ be the SID with $\ell-1$ vertices and a positive loop on each vertex. Then $\functions(H)$ contains a unique BN, which is the identity, thus $\maxpf(H)=2^{\ell-1}$. Since $D'$ is the disjoint union of $D$ and $H$, 
\[
\maxpf(D')=\maxpf(D)\cdot\maxpf(H)=\maxpf(D)\cdot 2^{\ell-1}.
\]  
Thus, $\maxpf(D')\geq 2^\ell\geq k$ if $\maxpf(D)\geq 2$, and $\maxpf(D')\leq 2^{\ell-1}<k$ otherwise.
\end{proof}

It remains to prove the case $k=2$: it is $\NP$-hard to decide if $\maxpf(D)\geq
2$. This is the main contribution of the paper. First, from a technical point
of view, the reduction used for this decision problem will be adapted for all
the other hardness results of this paper. Second, from a more general point of
view, many works have been devoted to the study of the SID of dynamical systems
with multiple steady states, both in the continuous and discrete setting, see
\cite{S06,RRT08,S06,KST07,R19} and the references therein. The basic
observation, answering a conjecture of the biologist Thomas, is that a
non-negative cycle must be present. The biological motivation behind is that
dynamical systems (in particular BNs) with multiple steady states (fixed points
in the discrete setting) should account for very important biological
phenomena: cell differentiation processes. In our setting, the principal
observation we just mention is the first item of Theorem
\ref{thm:basic_results}: if  $\maxpf(D)\geq 2$ then $D$ has a non-negative
cycle. This necessary condition for $\maxpf(D)\geq 2$ (which can be checked in
polynomial time) is obviously not sufficient, and it is natural to seek for a
characterization. By proving that it is NP-hard to decide that $\maxpf(D)\geq
2$, we show that any such characterization is difficult to check. Together with
the complexity upper bound established above, this answers a question stated in
\cite{R18}.   

\medskip
The rough idea is the following. We know that if \an{} SID $D_1$ has no source and
is full-positive then $\minpf(D_1)\geq 2$ (this is an easy consequence of the
last item of Theorem~\ref{thm:basic_results}). More precisely, every
$f\in\functions(D_1)$ has two distinct fixed points $x$ and $y$ such that
$x\leq y$. If, in addition, $D_1$ has a positive feedback vertex set of size
one, then $\maxpf(D_1)\leq 2$ thus every $f\in\functions(D_1)$ has {\em
exactly} two distinct fixed points $x$ and $y$ such that $x\leq y$. The
simplest example is the full-positive cycle. Now, if we add some negative arcs
in $D_1$, without producing additional positive cycles, this can only reduce
the maximum number of fixed points. In many cases the reduction is effective:
the resulting SID $D_2$ is such that $\maxpf(D_2)\leq 1$. The simplest example
is a full-positive cycle plus any negative arc. The idea is then to add to $D_2$
some additional sources that ``control'' the negative arcs in such a way that
some BNs behave as if their SID were $D_1$ (the negative arcs are ``dynamically
absent'') and have thus two fixed points, while some others behave as if their
SID  were $D_2$  (the negative arcs are ``dynamically present'') and have at
most one fixed point. If this ``control'' is possible if and only if rather
particular graphical conditions are satisfied, then some complexity lower bound
should be obtained. Actually, the reduction consists in encoding, in these
graphical conditions, a formula in conjunctive normal form.

\medskip
So consider a CNF formula $\psi$ over a set $\LAMBDA$ of $n$ variables and with a set of $m$ clauses $\MU$. To each variable $\lambda\in\LAMBDA$ is associated a positive literal $\lambda^+$ and a negative literal $\lambda^-$. The resulting sets of positive and negative literals are denoted $\LAMBDA^+$ and $\LAMBDA^-$, and each clause is regarded as a non-empty subset of $\LAMBDA^+\cup\LAMBDA^-$. If each clause is of size at most $3$, then $\psi$ is a $3$-CNF formula. An {\em assignment} for $\psi$ is regarded as a configuration $z$ on a set $V$ such that $\LAMBDA\subseteq V$. A positive literal $\lambda^+$ is satisfied by $z$ if $z_\lambda=1$, and a negative literal $\lambda^-$ is satisfied by $z$ if $z_\lambda=0$. A clause is satisfied by $z$ if at least one of its literals is satisfied by $z$. The formula $\psi$ is {\em satisfied} by $z$ (or $z$ is a {\em satisfying assignment} for $\psi$) if every clause in $\MU$ is satisfied by $z$. We say that $\psi$ is {\em satisfiable} if it has at least one satisfying assignment. 

\medskip
The reduction from the SAT problem is based on the following definition; see Figure~\ref{fig:D_psi} for an illustration. 

\begin{definition}[$D_\psi$]\label{def:D_psi}
Let $\psi$ be a CNF formula over a set $\LAMBDA$ of $n$ variables and with a set of $m$ clauses $\MU$. Given an enumeration $\LAMBDA=\{\lambda_1,\dots,\lambda_n\}$ of the variables and an enumeration $\MU=\{\mu_1,\dots,\mu_m\}$ of the clauses, we define \an{} SID $D_\psi$ with $4n+2m+1$ vertices as follows:
\begin{itemize}
\item
The vertex set is 
\[
V_\psi=\LAMBDA\cup U_\psi
\quad\text{with}\quad
U_\psi=\LAMBDA^+\cup\LAMBDA^-\cup\ELL\cup \MU\cup \C,
\]
where $\ELL=\{\ell_0,\dots,\ell_n\}$ and $\C=\{c_1,\dots,c_m\}$.  
\item The arcs are, for all $r\in [n]$ and $s\in [m]$, 
\begin{itemize}
\item
$(\lambda_r,\lambda^+_r),(\ell_{r-1},\lambda^+_r)$,
\item
$(\lambda_r,\lambda^-_r),(\ell_{r-1},\lambda^-_r)$, 
\item
$(\lambda^+_r,\ell_r),(\lambda^-_r,\ell_r)$,
\item
$(c_1,\ell_0)$,
\item
$(i,\mu_s)$ for all $i\in \LAMBDA^+\cup\LAMBDA^-$ such that $i$ is a literal in $\mu_s$,
\item
$(\mu_s,c_s),(c_{s+1},c_s)$, where $c_{m+1}$ means $\ell_n$.
\end{itemize} 
\item
For all $r\in [n]$ and $s\in [m]$, the arcs $(\lambda_r,\lambda^-_r)$ and $(\mu_s,c_s)$ are negative, and all the other arcs are positive. 
\end{itemize}
\end{definition}

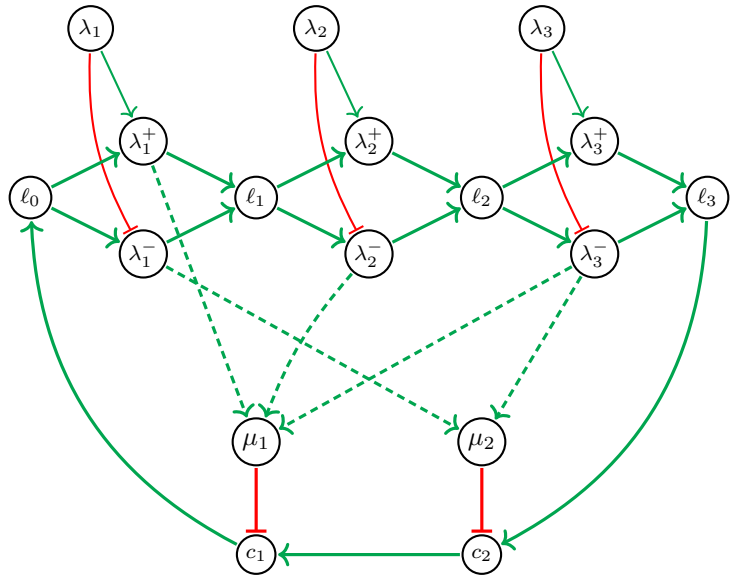
\begin{figure}
\[
\begin{tikzpicture}
\node[outer sep=1,inner sep=2,circle,draw,thick] (L1) at (0.8,2.5){\scriptsize$\lambda_1$};
\node[outer sep=1,inner sep=2,circle,draw,thick] (L2) at (3.8,2.5){\scriptsize$\lambda_2$};
\node[outer sep=1,inner sep=2,circle,draw,thick] (L3) at (6.8,2.5){\scriptsize$\lambda_3$};
\node[outer sep=1,inner sep=1,circle,draw,thick] (Lp1) at (1.5,1){\scriptsize$\lambda^+_1$};
\node[outer sep=1,inner sep=1,circle,draw,thick] (Lm1) at (1.5,-0.5){\scriptsize$\lambda^-_1$};
\node[outer sep=1,inner sep=1,circle,draw,thick] (Lp2) at (4.5,1){\scriptsize$\lambda^+_2$};
\node[outer sep=1,inner sep=1,circle,draw,thick] (Lm2) at (4.5,-0.5){\scriptsize$\lambda^-_2$};
\node[outer sep=1,inner sep=1,circle,draw,thick] (Lp3) at (7.5,1){\scriptsize$\lambda^+_3$};
\node[outer sep=1,inner sep=1,circle,draw,thick] (Lm3) at (7.5,-0.5){\scriptsize$\lambda^-_3$};
\node[outer sep=1,inner sep=2,circle,draw,thick] (1) at (0,0.25){\scriptsize$\ell_0$};
\node[outer sep=1,inner sep=2,circle,draw,thick] (2) at (3,0.25){\scriptsize$\ell_1$};
\node[outer sep=1,inner sep=2,circle,draw,thick] (3) at (6,0.25){\scriptsize$\ell_2$};
\node[outer sep=1,inner sep=2,circle,draw,thick] (4) at (9,0.25){\scriptsize$\ell_3$};
\node[outer sep=1,inner sep=2,circle,draw,thick] (mu1) at (3,-3){\small$\mu_1$};
\node[outer sep=1,inner sep=2,circle,draw,thick] (mu2) at (6,-3){\small$\mu_2$};
\node[outer sep=1,inner sep=2,circle,draw,thick] (c1) at (3,-4.5){\scriptsize$c_1$};
\node[outer sep=1,inner sep=2,circle,draw,thick] (c2) at (6,-4.5){\scriptsize$c_2$};
\path[Green,->,thick]
(L1) edge (Lp1)
(L2) edge (Lp2)
(L3) edge (Lp3)
(L1) edge[red, bend right=15,-|] (Lm1)
(L2) edge[red, bend right=15,-|] (Lm2)
(L3) edge[red, bend right=15,-|] (Lm3)
;
\path[Green,->,very thick]
(1) edge (Lp1)
(1) edge (Lm1)
(Lp1) edge (2)
(Lm1) edge (2)
(2) edge (Lp2)
(2) edge (Lm2)
(Lp2) edge (3)
(Lm2) edge (3)
(3) edge (Lp3)
(3) edge (Lm3)
(Lp3) edge (4)
(Lm3) edge (4)
(4) edge[bend left=30] (c2)
(c2) edge (c1)
(c1) edge[bend left=30] (1)
;
\path[Green,->,very thick]
(Lp1) edge[densely dashed] (mu1)
(Lm2) edge[densely dashed,bend right=10] (mu1)
(Lm3) edge[densely dashed] (mu1)
(mu1) edge[red,-|] (c1);
\path[Green,->,very thick]
(Lm1) edge[densely dashed] (mu2)
(Lm3) edge[densely dashed] (mu2)
(mu2) edge[red,-|] (c2)
;
\end{tikzpicture}
\]
\caption{\label{fig:D_psi}
SID $D_\psi$ of Definition~\ref{def:D_psi} for the 3-CNF formula $\psi=(\lambda_1\lor \neg\lambda_2\lor\neg\lambda_3)\land(\neg\lambda_1\lor \neg\lambda_3)$. Using our notations, the set of variables is $\LAMBDA=\{\lambda_1,\lambda_2,\lambda_3\}$ and the set of clauses is $\MU=\{\mu_1,\mu_2\}$ where $\mu_1=\{\lambda^+_1,\lambda^-_2,\lambda^-_3\}$ and $\mu_2=\{\lambda^-_1,\lambda^-_3\}$. Clauses are encoded in $D_\psi$ through dashed arrows. 
}
\end{figure}

Note that vertices in $\LAMBDA$ (variables) have in-degree $0$ (they are sources), vertices in $\MU$ (clauses) have in-degree at most the maximum size $k$ of a clause (hence $k\leq 2n$ and $k\leq 3$ if $\psi$ is a $3$-CNF formula), and all the other vertices have in-degree $2$, excepted $\ell_0$ which has in-degree $1$. 

\medskip
Intuitively, the SID $D_{\psi,1}$ obtained from $D_\psi$ by deleting the sources and the negative arcs corresponds to the SID $D_1$ in the above rough description of the construction: $D_{\psi,1}$ has no source and is full-positive, so that $\minpf(D_{\psi,1})\geq 2$, and since any vertex in $\ELL\cup\C$ is a positive feedback vertex set we have $\minpf(D_{\psi,1})=\maxpf(D_{\psi,1})=2$. Let $D_{\psi,2}$ obtained from $D_{\psi,1}$ by adding the negative arcs $(\mu_s,c_s)$. One can check that the addition of these negative arcs decreases the maximum number of fixed points: we have $\maxpf(D_{\psi,2})\leq 1$ (see \cite{R18}, Theorem 3). This SID $D_{\psi,2}$ plays the role of SID $D_2$ in the above rough description of the construction. 

\medskip
This construction works intuitively as follows. For the first direction, we consider an assignment $z\in\bool^{\LAMBDA}$ and we construct $f\in \functions(D_\psi)$,  in such a way that, at the first iteration, each source $\lambda_r$ is ``fixed'' to the state $z_{\lambda_r}$. By taking the OR function for the local functions associated with the positive and negative literals, it results that, at the second iteration, exactly one of $\lambda^+_r,\lambda^-_r$ is ``fixed'' to state $1$: the positive literal if $z_{\lambda_r}=1$, and the negative literal otherwise. At this point, if $z$ is a true assignment, each clause $\mu_s$ has at least one in-neighbor ``fixed'' to state $1$, and by taking $f_{\mu_s}$ as the OR function, $\mu_s$ is fixed at state $1$ at the third iteration. By taking $f_{c_s}$ as the OR function, each vertex $c_s$ is still ``free'', and by taking $f_{\ell_r}$ as the AND function, each vertex $\ell_r$ is still ``free''. Hence, there is a positive cycle containing only ``free'' vertices, which ``produces'' two fixed points. This is like if, in three iterations, all the negative arcs ``disappear'' leaving ``free'' a full-positive cycle. In this construction, only the local functions associated with the sources depend on the assignment $z$, and $f$ has at least one fixed point even if $z$ is not a true assignment. These two properties will be very useful for the other hardness results. Below, we do not described formally this ``three iterations process''. We only prove what we need for the following. 

\begin{lemma}\label{lem:psi_1}
Let $z\in\bool^{\LAMBDA}$ and $f\in\functions(D_\psi)$ defined by: $f_{\LAMBDA}=z$ and, for all $i\in U_\psi$, $f_i$ is the AND function if $i\in\ELL$ and the OR function otherwise. Then $f$ has at least one fixed point, and if $\psi$ is satisfied by $z$, then $f$ has at least two fixed points.
\end{lemma}

\begin{proof}
Let $x$ be the configuration on $V_\psi$ defined by $x_{\LAMBDA}=z$ and $x_{U_\psi}=1$. First, $f(x)_{\LAMBDA}=z=x_{\LAMBDA}$. Second, $f_{\ell_0}(x)=x_{c_1}=1=x_{\ell_0}$ and, for all $r\in [n]$, $f_{\ell_r}(x)=x_{\lambda^+_r}\land x_{\lambda^-_r}=1\land 1=1=x_{\ell_r}$. Third, any vertex $i\in U_\psi\setminus \ELL$ has a positive in-neighbor $j\in U_\psi$. Since $f_i$ is the OR function and $x_j=1$, we have $f_i(x)=1=x_i$. Thus, \label{key}$f(x)=x$. 

\medskip
Let $y$ be the configuration on $V_\psi$ be defined by: $y_{\LAMBDA}=z$ and, for all $i\in U_\psi$,
\[
y_i=1
\iff
\left\{
\begin{array}{l}
\textrm{$i$ is a clause $\mu_s$, or}\\
\textrm{$i$ is a positive literal $\lambda^+_r$ such that $z_{\lambda_r}=1$, or}\\
\textrm{$i$ is a negative literal $\lambda^-_r$ such that $z_{\lambda_r}=0$.}
\end{array}
\right.
\]
Note that $y_{\lambda^+_r}=z_{\lambda_r}\neq y_{\lambda^-_r}$ for all $r\in [n]$, so $x\neq y$. Suppose that $\psi$ is satisfied by $z$, and let us prove that $y$ is a fixed point of $f$. It is clear that $f(y)_{\LAMBDA}=z=y_{\LAMBDA}$. Let $r\in[n]$. Since exactly one of $y_{\lambda^+_r},y_{\lambda^-_r}$ is $0$, we have $f_{\ell_r}(y)=y_{\lambda^+_r}\land y_{\lambda^-_r}=0=y_{\ell_r}$; and $f_{\ell_0}(y)=y_{c_1}=0=y_{\ell_0}$. Then  
\[
\begin{array}{ll}
f_{\lambda^+_r}(y)=y_{\lambda_r}\lor y_{\ell_{r-1}}=z_{\lambda_r}\lor 0=z_{\lambda_r}=y_{\lambda^+_r},\\[1mm]
f_{\lambda^-_r}(y)=\neg y_{\lambda_r}\lor y_{\ell_{r-1}}=\neg z_{\lambda_r}\lor 0=\neg z_{\lambda_r}=y_{\lambda^-_r}.
\end{array}
\]
Let $s\in [m]$. Since $\psi$ is satisfied by $z$, $\mu_s$ has at least one in-neighbor which is a positive literal $\lambda^+_r$ with $z_{\lambda_r}=1$ or a negative literal $\lambda^-_{\lambda_r}$ such that $z_{\lambda_r}=0$. Hence, $\mu_s$ has an in-neighbor $i$ with $y_i=1$ and since $f_{\mu_s}$ is the OR function we deduce that $f_{\mu_s}(y)=1=y_{\mu_s}$. We finally prove that $f_{c_s}(y)=0=y_{c_s}$ by induction on $s$ from $m+1$ to $1$. Since $c_{m+1}$ means $\ell_n$, the case $s=m+1$ is already proven. Then, for $s\in[m]$, we have $y_{c_{s+1}}=0$ by induction so $f_{c_s}(y)=\neg y_{\mu_s}\lor y_{c_{s+1}}=\neg 1\lor 0=0=y_{c_s}$. Thus, $f(y)=y$. 
\end{proof}

For the other direction, we suppose that there is $f\in \functions(D_\psi)$ with two distinct fixed points $x$ and $y$. Then we prove that either $x\leq y$ or $y\leq x$, exactly as if the SID of $f$ where $D_{\psi,1}$ (or any full-positive SID without two vertex-disjoint positive cycles); this results from Lemma~\ref{lem:pseudo_monotone} below. Furthermore, for each clause $\mu_s$ it appears that $x_{\mu_s}=y_{\mu_s}$, the clause is ``fixed''. Since $x\leq y$ or $y\leq x$, this is possible only if $\mu_s$ has an in-neighbor (a positive or negative literal associated to some variable $\lambda_r$) which is also ``fixed''; this results from Lemma~\ref{lem:constant_vertex} below. Beside, an easy consequence of $x\neq y$ is that at most one $\lambda^+_r,\lambda^-_r$ is ``fixed'' and from that point we easily obtain a satisfying assignment. 


\begin{lemma}\label{lem:pseudo_monotone}
Let $D$ be a \simple{} SID, without two vertex-disjoint positive cycles, such that every positive cycle is full-positive and, for any negative arc $(j,i)$, either $j$ is a source or every positive cycle contains $i$. Let $f\in\functions(D)$ with two fixed point $x$ and $y$.~The following~holds: 
\begin{itemize}
\item We have $x\leq y$ or $y\leq x$. 
\item If $(j,i)$ is a negative arc and $i$ is of in-degree two, then $x_j=y_j$.   
\end{itemize}
\end{lemma}

\begin{proof}
Let $I$ be the set of vertices $i$ in $D$ such that $x_i\neq y_i$. We prove
that each $i\in I$ has an in-neighbor $j\in I$ such that
$x_j\xor\tilde\sigma_{ji}=x_i$. Let $i\in I$ and suppose that $x_i<y_i$, the
other case being similar. Then $f_i(x)<f_i(y)$ and thus $y\not\leq^D_i x$ by
Lemma~\ref{lem:partial_order}. Hence, $i$ has at least one in-neighbor $j$
such that: $(j,i)$ is positive and $x_j<y_j$, or $(j,i)$ is negative and
$x_j>y_j$. In both cases, $j\in I$ and $x_j\xor\tilde\sigma_{ji}=x_i$. 

\medskip
We deduce that $D$ has a subgraph $D'$ with vertex set $I$, in which each vertex is of in-degree one, and such that $x_j\xor\tilde\sigma_{ji}=x_i$ for any arc $(j,i)$ of $D'$. Hence, for any $i,j\in I$, if $x_j=x_i$ then any walk in $D'$ from $j$ to $i$ visits an even number of negative arcs. In particular, any cycle of $D'$ is positive. Thus, by hypothesis, any cycle of $D'$ is full-positive.  

\medskip
Since, in $D'$, each vertex is of in-degree one, $D'$ has a cycle $C$ (which is full-positive). Furthermore, $D'$ has no other cycle (since otherwise $D$ has two vertex-disjoint positive cycles). Thus, $D'$ is connected. Furthermore, $D'$ is full-positive: if $(j,i)$ is a negative arc of $D'$ then $j$ is not a source (since $j\in I$) so, by hypothesis, $i$ is in $C$, and thus $j$ too. But then $C$ has a negative arc, and it is a contradiction. So $D'$ is full-positive and we deduce that $x_j=x_i$ for any arc $(j,i)$ of $D'$. Since $D'$ is connected, we deduce that $x_j=x_i$ for any $i,j\in I$ and this implies that $y_j=y_i$ for any $i,j\in I$. Thus, either $x\leq y$ or $y\leq x$. Suppose, without loss, that $x\leq y$. 

\medskip
Let $(j,i)$ be a negative arc of $D$. Suppose that $i$ is of in-degree two in $D$ and, for a contradiction, that $x_j\neq y_j$. Then $x_j<y_j$ since $x\leq y$. By hypothesis, $i$ is in $C$ thus $x_i<y_i$. Since $i$ is of in-degree two, $f_i$ is either the AND function or the OR function. If $f_i$ is the OR function, then $f_i(x)=1$ since $x_j=0$ and $(j,i)$ is negative, which is a contradiction since $x_i=0$. If $f_i$ is the AND function, then $f_i(y)=0$ since $y_j=1$ and $(j,i)$ is negative, which is a contradiction since $y_i=1$. Thus, $x_j=y_j$.  
\end{proof}

\begin{lemma}\label{lem:constant_vertex}
Let $D$ be \an{} SID and $f\in\functions(D)$ with two fixed points $x,y$ such that $x\leq y$. Every vertex $i$ with only positive in-neighbors such that $x_i=y_i$ has a positive in-neighbor $j$ such that $x_j=y_j$.
\end{lemma}

\begin{proof}
Let $i$ be as in the statement. If $x_i=y_i=0$ then, by
Lemma~\ref{lem:partial_order}, $i$ has a positive in-neighbor $j$ such that
$y_j=f_i(y)=y_i=0$, and since $x\leq y$ we have $x_j=y_j=0$. If $x_i=y_i=1$
then, by Lemma~\ref{lem:partial_order}, $i$ has a positive
in-neighbor $j$ such that $x_j=f_i(x)=x_i=1$, and since $x\leq y$ we have
$x_j=y_j=1$. 
\end{proof}

We go back to our construction $D_\psi$ and prove a kind of converse of Lemma \ref{lem:psi_1}. We need a definition. For every $r\in [n]$, $\ell_r$ is of in-degree two in $D_\psi$, so $f_{\ell_r}$ is either the AND function or the OR function; we then define $\epsilon(f)\in\bool^{\LAMBDA}$ as follows: for all $r\in [n]$, 
\[
\epsilon(f)_{\lambda_r}=
\left\{
\begin{array}{ll}
0&\textrm{if $f_{\ell_r}$ is the OR function,}\\[1mm]
1&\textrm{if $f_{\ell_r}$ is the AND function.}
\end{array}
\right.
\]

\begin{lemma}\label{lem:psi_2}
If $f\in\functions(D_\psi)$ has distinct fixed points $x$ and $y$, then $\psi$ is satisfied by $x_{\LAMBDA}\xor\epsilon(f)$. 
\end{lemma}

\begin{proof}
Let $f\in\functions(D_\psi)$ and $\epsilon=\epsilon(f)$. Suppose that $f$ has distinct fixed points $x$ and $y$. By Lemma~\ref{lem:pseudo_monotone}, we have $x\leq y$ or $y\leq x$. Suppose, without loss, that $x\leq y$. Since vertices in $\LAMBDA$ are sources, there is $z\in\bool^{\LAMBDA}$ such that $x_{\LAMBDA}=y_{\LAMBDA}=z$. Note also that, for every $i\in\ELL$, $\{i\}$ is a positive feedback vertex, and thus $x_i<y_i$ by Lemma \ref{lem:positive_feedback_lemma}.

\medskip
Consider any clause $\mu_s$, and let us prove that it is satisfied by $z\oplus\epsilon$. By Lemma \ref{lem:pseudo_monotone}, we have $x_{\mu_s}=y_{\mu_s}$ and, by Lemma~\ref{lem:constant_vertex}, $\mu_s$ has an in-neighbor $i$ such that $x_i=y_i$. This in-neighbor $i$ is a positive or negative literal contained in $\mu_s$ and associated with some variable, say $\lambda_r$. We prove that this literal is satisfied by $z\oplus \epsilon$, and so is $\mu_s$. We consider two cases:
\begin{enumerate}
\item
Suppose that $i=\lambda^+_r$. Since $x_{\ell_{r-1}}<y_{\ell_{r-1}}$, if $f_i$ is the OR function we have 
\[
x_i=f_i(x)=x_{\lambda_r}\lor x_{\ell_{r-1}}=x_{\lambda_r}\lor 0= x_{\lambda_r}=z_{\lambda_r},
\]
and if $f_i$ is the AND function we have 
\[
y_i=f_i(y)=y_{\lambda_r}\land y_{\ell_{r-1}}=y_{\lambda_r}\land 1= y_{\lambda_r}=z_{\lambda_r}.
\]
Since $x_i=y_i$, we have $z_{\lambda_r}=x_i=y_i$ and consider two cases.
	\begin{enumerate}	
	
	\item 
	If $z_{\lambda_r}=1$ then $x_i=1$ so, if $f_{\ell_r}$ is the OR function, then 
	\[
	x_{\ell_r}=f_{\ell_r}(x)=x_i\lor x_{\lambda^-_r}=1\lor x_{\lambda^-_r}=1,
	\]
	and this contradicts $x_{\ell_r}<y_{\ell_r}$. Hence, $f_{\ell_r}$ is the AND function so $\epsilon_{\lambda_r}=0$. 

	\item 
	If $z_{\lambda_r}=0$ then $y_i=0$ so, if $f_{\ell_r}$ is the AND function, then 
	\[
	x_{\ell_r}=f_{\ell_r}(y)=y_i\land y_{\lambda^-_r}=0\land y_{\lambda^-_r}=0,
	\]
	and this contradicts $x_{\ell_r}<y_{\ell_r}$. Hence, $f_{\ell_r}$ is the OR function so $\epsilon_{\lambda_r}=1$. 

	\end{enumerate}
In both cases we have $z_{\lambda_r}\xor\epsilon_{\lambda_r}=1$ thus $i=\lambda^+_r$ is satisfied by $z\oplus \epsilon$.
\item
Suppose that $i=\lambda^-_r$. Since $x_{\ell_{r-1}}<y_{\ell_{r-1}}$, if $f_i$ is the OR function we have 
	\[
	x_i=f_i(x)=\neg x_{\lambda_r}\lor x_{\ell_{r-1}}=\neg x_{\lambda_r}\lor 0= \neg x_{\lambda_r}=\neg z_{\lambda_r},
	\]
and if $f_i$ is the AND function we have 
	\[
	y_i=f_i(y)=\neg y_{\lambda_r}\land y_{\ell_{r-1}}=\neg y_{\lambda_r}\land 1= \neg y_{\lambda_r}=\neg z_{\lambda_r}.
	\]
We deduce that $z_{\lambda_r}\neq x_i=y_i$. If $z_{\lambda_r}=0$ then $x_i=1$ and we deduce as in the first case that $\epsilon_{\lambda_r}=0$. If $z_i=1$ then $y_i=0$ and we deduce as in the first case that $\epsilon_{\lambda_r}=1$. In both cases we have $z_{\lambda_r}\xor\epsilon_{\lambda_r}=0$ thus $i=\lambda^-_r$ is satisfied by $z\xor\epsilon$.  
\end{enumerate}
\end{proof}

We can now prove the main property of $D_\psi$. 

\begin{lemma}\label{lem:psi_main}
\[
\maxpf(D_\psi)=
\left\{
\begin{array}{ll}
2&\textrm{if $\psi$ is satisfiable},\\
1&\textrm{otherwise.}
\end{array}
\right.
\]
\end{lemma}

\begin{proof}
By Lemma~\ref{lem:psi_1} and the positive feedback bound, we have $1\leq \maxpf(D_\psi)\leq 2$. By Lemma \ref{lem:psi_2}, if $\maxpf(D_\psi)=2$ then $\psi$ is satisfiable, and the converse is given by Lemma~\ref{lem:psi_1}.
\end{proof}

The following lemma is useful to reduce the maximum in-degree without modifying the maximum and minimum number of fixed points. The proof is a simple exercise and is omitted. The transformation described in the statement is illustrated Figure~\ref{fig:delta_reduction}.

\begin{lemma}\label{lem:delta_reduction}
Let $D$ be a SID and let $I$ be the set of vertices with at least three in-neighbors. Suppose that each vertex in $I$ has only positive in-neighbors. Let $D'$ be obtained from $D$ by doing the following operations for each $i\in I$. Denoting $j_1,\dots,j_k$ the in-neighbors of $i$,  
\begin{itemize}
\item
we remove the arc from $j_\ell$ to $i$ for $1\leq \ell < k$, 
\item
we add a new vertex $j'_\ell$ and a positive arc from $j_\ell$ to $j'_\ell$ for $1<\ell<k$,  and
\item
setting $j'_1=j_1$ and $j'_{k}=i$, we add a positive arc from $j'_\ell$ to $j'_{\ell+1}$ for $1\leq \ell<k$.
\end{itemize}
Then $\Delta(D')\leq 2$ and $D$ has at most $|V_D|+|I|\Delta(D)-2|I|$ vertices. Furthermore if there is $f \in \functions(D)$ with $\maxpf(D)$ fixed points such that $f_i$ is the OR function for all $i \in I$, then $\maxpf(D')=\maxpf(D)$. Similarly, if there is $f \in \functions(D)$ with $\minpf(D)$ fixed points such that $f_i$ is the OR function for all $i \in I$, then $\minpf(D')=\minpf(D)$.
\end{lemma}

\begin{figure}
\[
\begin{array}{c}
\begin{tikzpicture}
\node[outer sep=1,inner sep=2,circle,draw,thick] (j1) at (0,2){$j_1$};
\node[outer sep=1,inner sep=2,circle,draw,thick] (j2) at (1.5,2){$j_2$};
\node[outer sep=1,inner sep=2,circle,draw,thick] (j3) at (3,2){$j_3$};
\node[outer sep=1,inner sep=2,circle,draw,thick] (j4) at (4.5,2){$j_4$};
\node[outer sep=1,inner sep=2,circle,draw,thick] (i) at (2.25,0){$i$};
\path[Green,->,thick]
(j1) edge (i)
(j2) edge (i)
(j3) edge (i)
(j4) edge (i)
;
\end{tikzpicture}
\end{array}
\quad\mapsto\quad
\begin{array}{c}
\begin{tikzpicture}
\node[outer sep=1,inner sep=2,circle,draw,thick] (j1) at (0,2){$j_1$};
\node[outer sep=1,inner sep=2,circle,draw,thick] (j2) at (1.5,2){$j_2$};
\node[outer sep=1,inner sep=2,circle,draw,thick] (j3) at (3,2){$j_3$};
\node[outer sep=1,inner sep=2,circle,draw,thick] (j4) at (4.5,2){$j_4$};
\node[outer sep=1,inner sep=2,circle,draw,thick] (j'2) at (1.5,0){$j'_2$};
\node[outer sep=1,inner sep=2,circle,draw,thick] (j'3) at (3,0){$j'_3$};
\node[outer sep=1,inner sep=2,circle,draw,thick] (i) at (4.5,0){$i$};
\path[Green,->,thick]
(j1) edge (j'2)
(j2) edge (j'2)
(j3) edge (j'3)
(j4) edge (i)
(j'2) edge (j'3)
(j'3) edge (i)
;
\end{tikzpicture}
\end{array}
\]
\caption{\label{fig:delta_reduction} Transformation in Lemma~\ref{lem:delta_reduction} used to reduce the maximum in-degree. 
}
\end{figure}
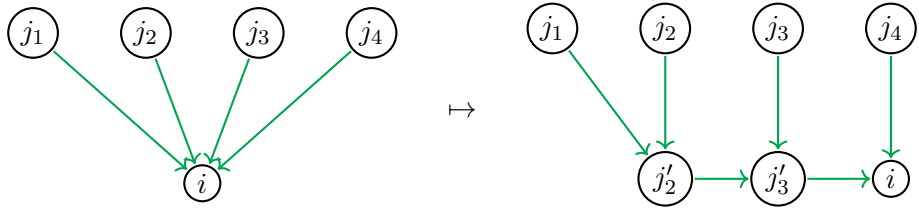

We then deduce the $\NP$-hardness of $\kmaxfpp{2}$. 

\begin{lemma}\label{lem:NP-hard}
\kmaxfpp{2} is $\NP$-hard, even when restricted to SIDs $D$ such that $\Delta(D)\leq 2$.
\end{lemma}

\begin{proof}
Given a 3-SAT formula $\psi$ with $n$ variables and $m$ clauses, $D_\psi$ has $4n+2m+1$ vertices and, by Lemma~\ref{lem:psi_main}, we have $\maxpf(D_\psi)\geq 2$ if and only if $\psi$ is satisfiable. We have $\Delta(D_\psi)\leq 3$ and since vertices than can have a in-degree of at least three correspond to the $m$ clauses and have only positive in-neighbors, the SID $D'_\psi$ obtained from $D_\psi$ as in Lemma~\ref{lem:delta_reduction} has at most $4n+3m+1$ vertices and $\Delta(D'_\psi)\leq 2$. By Lemmas~\ref{lem:psi_1} and \ref{lem:psi_main}, there is $f\in\functions(D_\psi)$ with $\maxpf(D_\psi)$ fixed points such that $f_i$ is the OR function for every vertex $i$ of in-degree at least three. Hence, by Lemma~\ref{lem:delta_reduction}, we have $\maxpf(D'_\psi)=\maxpf(D_\psi)$ and the lemma is proven.  
\end{proof}

The proof of Theorem~\ref{thm:NP-complete} is an obvious consequence of Lemmas~\ref{lem:NP}, \ref{lem:k_to_2} and \ref{lem:NP-hard}. 

\begin{remark}
Let $D''_\psi$ obtained from $D'_\psi$ by adding: two new vertices $u,v$; a positive arcs $(\ell_0,u)$; a negative arc $(\ell_0,v)$; and, for every $r\in [n]$, the positive arcs $(u,\lambda_r),(v,\lambda_r)$. One can check that these operations does not change the maximum number of fixed points, that is, $\maxpf(D''_\psi)=\maxpf(D'_\psi)$. Since $D''_\psi$ is strongly connected this shows that \kmaxfpp{2} is $\NP$-hard, even when restricted to strongly connected SIDs $D$ with $\Delta(D)\leq 2$. By adapting conveniently Lemma~\ref{lem:k_to_2}, we obtain that, for all $k\geq 2$, \kmaxfpp{k} is $\NP$-hard, even when restricted to strongly connected SIDs $D$ such that $\Delta(D)\leq 2$. 
\end{remark}

\subsection{Extensions}

We now present easy extensions of the key Lemmas \ref{lem:psi_1} and \ref{lem:psi_2},
which will be  useful for the other hardness results. We start with some definitions. 

\begin{definition}[Extensions of $D_\psi$ and partial fixed points]\label{def:extension}
An {\em extension} of $D_\psi$ is \an{} SID $D$ which is the union of $D_\psi$ and \an{} SID $H$ with $V_H\cap V_\psi\subseteq \LAMBDA$; equivalently:  
\begin{itemize}
\item
$D_\psi$ is a subgraph of $D$ (\ie{} each arc of $D_\psi$ is in $D$ with the same sign), and 
\item
in- and out-neighbors of vertices in $U_\psi$ are identical in $D$ and $D_\psi$. 
\end{itemize}
Let $D$ be an extension of $D_\psi$ and $I=V_D\setminus U_\psi$. Given $f\in\functions(D)$, we say that a configuration $z$ on $I$ is a {\em partial fixed point} of $f$ if $f(x)_I=x_I=z$ for some configuration $x$ on $V_D$ (which thus extends $z$); since there is no arc from $U_\psi$ to $I$, if $z$ is a partial fixed point, then $f(x)_I=z$ for {\em every} configuration $x$ on $V_D$ extending $z$.
\end{definition}

Note that every fixed point extends a partial fixed point. Note also that
$D_\psi$ is an extension of itself with $I=\LAMBDA$, and each
$f\in\functions(D_\psi)$ has a unique partial fixed point: the configuration
$z$ on $\LAMBDA$ such that $f_{\LAMBDA}=z$. If $f$ is as in
Lemma~\ref{lem:psi_1} and $\psi$ is satisfied by $z$, then $z$ can be extended
into two (global) fixed points. Conversely, in any case, by
Lemma~\ref{lem:psi_2}, if $z$ can be extended into two (global) fixed points,
then $\psi$ is satisfied by $z\xor\epsilon(f)$. The next two lemmas show that,
more generally, if the SID of $f$ is an extension of $D_\psi$, then these
properties remain true for {\em every} partial fixed point $z$.  

\begin{lemma} \label{lemma:D_psi_inf}\label{lem:ext1}
Let $D$ be an extension of $D_\psi$. Let $f\in\functions(D)$ such that, for all $i\in U_\psi$, $f_i$ is the AND function if $i\in\ELL$ and the OR function otherwise. Let $z$ be a partial fixed point of~$f$. Then $f$ has at least one fixed point extending $z$ and, if $\psi$ is satisfied by $z_{\LAMBDA}$, then $f$ has at least two fixed points extending $z$. 
\end{lemma}

\begin{proof}
Let $I=V_D\setminus U_\psi$ and let $\tilde f\in\functions(D_\psi)$ be defined by $\tilde f_{\LAMBDA}=z_{\LAMBDA}$ and, for all $i\in U_\psi$, $\tilde f_i$ is the AND function if $i\in\ELL$ and the OR function otherwise. By Lemma~\ref{lem:psi_1}, $\tilde f$ has a fixed point $\tilde x$. Let $x$ be the configuration on $V_D$ defined by $x_I=z$ and $x_{U_\psi}=\tilde x_{U_\psi}$.  Since $z$ is a partial fixed point, we have $f(x)_I=x_I$, and since each vertex $i\in U_\psi$ has the same in-neighbors in $D$ and $D_\psi$, and since $f_i$ and $\tilde f_i$ are either both AND functions or both OR functions, we have $f_i(x)=\tilde f_i(\tilde x)=\tilde x_i=x_i$. Thus, $x$ is a fixed point of $f$. By Lemma~\ref{lem:psi_1}, if $\psi$ is satisfied by $z_{\LAMBDA}$, then $\tilde f$ has a fixed point $\tilde y\neq \tilde x$. Similarly, the configuration $y$ on $V_D$ defined by $y_I=z$ and $y_{U_\psi}=\tilde y_{U_\psi}$ is a fixed point of $f$. Since $\tilde y_{\LAMBDA}=\tilde x_{\LAMBDA}=z_{\LAMBDA}$, we have $\tilde x_{U_\psi}\neq \tilde y_{U_\psi}$ and thus $x\neq y$. 
\end{proof}

If $D$ is an extension of $D_\psi$, each vertex $\ell_r$ is of in-degree two in $D$ thus, for each $f\in \functions(D)$, $f_{\ell_r}$ is either the AND function or the OR function. Then, as previously, we define $\epsilon(f)\in \bool^{\LAMBDA}$ by $\epsilon(f)_{\lambda_r}=0$ if $f_{\ell_r}$ is the OR function and $\epsilon(f)_{\lambda_r}=1$ otherwise. 

\begin{lemma}\label{lem:ext2}
Let $D$ be an extension of $D_\psi$. Choose $f\in\functions(D)$ and let $z$ be a partial fixed point of $f$. Then $f$ has at most two fixed points extending $z$, and if $f$ has two fixed points extending $z$, then $\psi$ is satisfied by $z_{\LAMBDA}\xor\epsilon(f)$. 
\end{lemma}

\begin{proof}
Let $I=V_D\setminus U_\psi$. Since $z\in\bool^I$ and $D\setminus I$ has a positive feedback vertex set of size one, by Lemma~\ref{lem:positive_feedback_lemma}, $f$ has at most two fixed points extending $z$. Let $\tilde f$ be the BN with component set $V_\psi$ defined by: $\tilde f_{\LAMBDA}=z_{\LAMBDA}$ and for all configurations $x$ on $V_D$, $\tilde f(x_{V_\psi})_{U_\psi}=f(x)_{U_\psi}$; there is no ambiguity since all the in-neighbors of vertices in $U_\psi$ are in $V_\psi$. Since vertices in $\LAMBDA$ are sources of $D_\psi$ and since every vertex in $U_\psi$ have exactly the same in-neighbors in $D$ and $D_\psi$ it is clear that $\tilde f\in \functions(D_\psi)$. Suppose now that $f$ has two distinct fixed points $x,y$ extending $z$. We easily check that the restrictions $\tilde x,\tilde y$ of $x,y$ on $V_\psi$ are fixed points of $\tilde f$ and that $\epsilon(\tilde f)=\epsilon(f)$. Since $x_I=y_I=z$ and $x\neq y$, we have $\tilde x\neq\tilde y$ thus, by Lemma \ref{lem:psi_2}, $\psi$ is satisfied by $\tilde x_{\LAMBDA}\xor \epsilon(\tilde f)=z_{\LAMBDA}\xor \epsilon(f)$.
\end{proof}

\section{\MaxFPP{}}\label{section:max_k>=k}

In this section, we study the case where $k$ is not a constant but a parameter of the problem. 
Contrary to the two previous sections, we will see here that the restriction to SIDs with a bounded maximum in-degree reduces the complexity of the problem.

\subsection{{\boldmath$\Delta(D)$\unboldmath} bounded} \label{subsection:borne}

In this first subsection, we study the problem \maxfpp{} for the family of SIDs with a maximum in-degree bounded by a constant $d\geq 2$, and prove that it is $\NPoSPoly$-complete. To introduce this complexity class, let us first recall that problems in $\SPoly$ consist in counting the number of certificates of decision problems in $\NP$ (the number of accepting branches of a non-deterministic polynomial time algorithm). The canonical $\SPoly$-complete problem is \sharpSAT{}, which asks for the number of satisfying assignments of an input $3$-CNF formula. The class $\NPoSPoly$ then corresponds to decision problems computable in polynomial time by a non-deterministic Turing machine with an oracle in $\SPoly$ (a ``black box'' answering any problem in the class $\SPoly$ without using any resource). 

\begin{theorem} \label{theorem:mfpp-Delta_bounded}
When $\Delta(D)\leq d$, \maxfpp{} is $\NPoSPoly$-complete.
\end{theorem}

We first prove the upper bound. 

\begin{lemma}\label{lemma:mfpp-Delta_bounded_in}
When $\Delta(D)\leq d$, \maxfpp{} is in $\NPoSPoly$.
\end{lemma}

\begin{proof}
Let $d\geq 2$ be a fixed integer. Consider the algorithm, which takes as input \an{} SID $D$ with $\Delta(D)\leq d$ and an integer $k$, and proceeds as follows. 
	\begin{enumerate}
	\item 
	Choose non-deterministically a BN $f$ with component set $V_D$ such that, for all $i\in V_D$, the local function $f_i$ only depends on components in $N(i)$; this can be done in linear time since each local function $f_i$ can be represented using $2^{|N(i)|}\leq 2^d$ bits. 
	\item 
	Compute the SID $D_f$ of $f$; this can be done in quadratic time since, to compute the in-neighbors of each $i\in V_D$ and the corresponding signs, we only have to consider $2^{|N(i)|}\leq 2^d$ configurations (for each configuration $x$ on $N(i)$ and $j\in N(i)$, we compare $f_i(\tilde x)$ and $f_i(\tilde x\xor e_j)$ where $\tilde x$ is any configuration on $V_D$ extending $x$). 
	\item 
	Compute $\phi(f)$, the number of fixed points of $f$, with a call to the $\SPoly$ oracle (the problem of deciding if $f$ has a fixed point is trivially in $\NP$: choose non-deterministically a configuration, and accept if and only if it is a fixed point).
	\item 
	Accept if and only if $\phi(f) \geq k$ and $D_f=D$.
	\end{enumerate}
This non-deterministic polynomial time algorithm has an accepting branch if and only $\maxpf(D) \geq k$, so \maxfpp{} is in $\NPoSPoly$. 
\end{proof}

For the lower bound, it is convenient to define $\NPoSPoly$ in another way. The class $\PP$ regroups decision problems decided by a probabilistic Turing machine in polynomial time, with a probability of error less than a half. 
The canonical $\PP$-complete decision problem is \MAJSAT{}, which asks if the majority of the assignments of a given CNF formula are satisfying. The class $\NPoPP$ then corresponds to decision problems computable in polynomial time by a non-deterministic Turing machine with an oracle in  $\PP$. It is well known that $\PolyoSPoly = \PolyoPP$~\cite{P94}, and from that we easily deduce that $\NPoSPoly = \NPoPP$. The following problem is known to be $\NPoPP$-complete~\cite{MJ98,AW21}, and thus also $\NPoSPoly$-complete.  

\begin{quote}\decisionpb
	{\EMSAT ~(\emsat)}
	{a CNF formula $\psi$ over  $\LAMBDA=\{\lambda_1,\dots,\lambda_n\}$ and $s\in [n]$.}
	{does there exist $z'\in\bool^{\LAMBDA'}$, where $\LAMBDA'=\{\lambda_1,\dots,\lambda_s\}$, such that $\psi$ is satisfied by the majority of the assignments $z\in\bool^{\LAMBDA}$ extending $z'$?}
\end{quote}

\begin{lemma}
	\label{lemma:mfpp-Delta_bounded_hard}
	When $\Delta(D)\leq d$, \maxfpp{} is $\NPoSPoly$-hard.
\end{lemma}

\begin{proof}
We present a reduction from \emsat{}. Let $\psi$ be a CNF formula over the set of variables $\LAMBDA=\{\lambda_1,\dots,\lambda_n\}$. Let $s\in [n]$ and $\LAMBDA'=\{\lambda_1,\dots,\lambda_s\}$. Given $z'\in\bool^{\LAMBDA'}$, we denote by $E(z')$ the $2^{n-s}$ assignments $z\in\bool^{\LAMBDA}$ extending $z'$. Then we denote by $\alpha(z')$ the number of $z\in E(z')$ satisfying $\psi$, and $\alpha^*$ is the maximum of $\alpha(z')$ for $z'\in\bool^{\LAMBDA'}$. Thus, $(\psi,s)$ is a true instance of \emsat{} if and only if $\alpha^*\geq 2^{n-s-1}$. 
	
\medskip	
Let $D_{\psi,s}$ obtained from $D_\psi$ by adding, for every $i\in\LAMBDA\setminus \LAMBDA'$, a positive loop on $i$. Thus, $D_{\psi,s}$ is an extension of $D_\psi$. Let us prove that:
\[
\maxpf(D_{\psi,s}) = 2^{n-s} +\alpha^*.
\]

\medskip
Let $z'\in\bool^{\LAMBDA'}$ such that $\alpha(z')=\alpha^*$. Let $f\in\functions(D_{\psi,s})$ such that $f_{\LAMBDA'}=z'$ and, for all $i\in U_\psi$, $f_i$ is the AND function if $i\in\ELL$ and the OR function otherwise. Since each vertex in $ \LAMBDA\setminus \LAMBDA'$ has a positive loop and no other in-coming arc, $E(z')$ is the set of partial fixed points of $f$. Hence, by Lemma~\ref{lem:ext1}, for every $z\in E(z')$, $f$ has a fixed point extending $z$ and, if $\psi$ is satisfied by $z$, then $f$ has another fixed point extending $z$. We deduce that $f$ has at least $2^{n-s}+\alpha(z')$ fixed points, and thus $\maxpf(D_{\psi,s}) \geq  2^{n-s} +\alpha(z')=2^{n-s}+\alpha^*$.  

\medskip
For the other direction, choose any $f\in\functions(D_{\psi,s})$ and set $\epsilon=\epsilon(f)$. Since vertices in $\LAMBDA'$ are sources, $f_{\LAMBDA'}=z'$ for some $z'\in\bool^{\LAMBDA'}$. Since each vertex in $ \LAMBDA\setminus \LAMBDA'$ has a positive loop and no other in-coming arc, $E(z')$ is the set of partial fixed points of $f$. By Lemma \ref{lem:ext2}, for every $z\in E(z')$, $f$ has at most two fixed points extending $z$, and if $f$ has two fixed points extending $z$, then $\psi$ is satisfied by $z\xor\epsilon$. Since every fixed point extends a partial fixed point, we deduce that $f$ has at most $2^{n-s}+\beta$ fixed points, where $\beta$ is the number of $z\in E(z')$ such that $z\xor\epsilon$ is satisfying. Since $\beta=\alpha(z'+\epsilon_{\LAMBDA'})\leq\alpha^*$, we deduce that $f$ has at most $2^{n-s} +\alpha^*$ fixed points. Hence,  $\maxpf(D_{\psi,s}) \leq  2^{n-s} +\alpha^*$. 

\medskip
Consequently, we have $\alpha^*\geq 2^{n-s-1}$ if and only if $\maxpf(D_{\psi,s})\geq 3\cdot 2^{n-s-1}$. Thus, $(\psi,s)$ is a true instance of \emsat{} if and only if $(D_{\psi,s},k)$ is a true instance of \maxfpp{}, where $k=3\cdot 2^{n-s-1}$. 

\medskip
We have $\Delta(D_{\psi,s})\leq 2n$ but, using Lemma~\ref{lem:delta_reduction}, we can obtain from $D_{\psi,s}$ \an{} SID $D'_{\psi,s}$ with $\Delta(D'_{\psi,s})\leq 2$ and at most $4n+2nm+1$ vertices such that $\maxpf(D'_{\psi,s})=\maxpf(D_{\psi,s})$ (for that we use the fact, showed above, that there is $f\in\functions(D_{\psi,s})$ with $\maxpf(D_{\psi,s})$ fixed points where $f_i$ is the OR function for every vertex $i$ of in-degree at least three). 
\end{proof}

\subsection{\boldmath$\Delta(D)$ unbounded\unboldmath} \label{subsection:non borne}

In this second subsection, we study \maxfpp{} without restriction on the maximum in-degree, and prove the following.

\begin{theorem} \label{theorem:mfpp_unbounded}
\maxfpp{} is $\NEXPTIME$-complete.
\end{theorem}

In Section~\ref{section:definition} we proved that \maxfpp{} is in $\NEXPTIME$, so we only have to prove its $\NEXPTIME$-hardness; this is done with a reduction from \SSAT{}, which is a basic complete problem for this class  (see Theorem 20.2 of~\cite{P94}).

\medskip
\SSAT{} takes as input a succinct representation of a 3-CNF formula $\Psi$, under the form of a circuit, and asks if $\Psi$ is satisfiable. To give details, we need to introduce circuits and, to work in a single framework, it is convenient to introduce circuits in the language of BNs.

\medskip
Let $C$ be \an{} SID obtained from an acyclic SID by adding a positive loop on some
sources. Let $I$ be the set of vertices with a positive loop, called {\em input
vertices}, and let $O$ the set of vertices of out-degree zero, called {\em
output vertices}. We say that $C$ is a {\em circuit structure}. Then any
$h\in\functions(C)$ is a {\em circuit} that encodes a map $g$ from $\bool^I$ to
$\bool^O$ as follows. Since input vertices have a positive loop and no other
in-coming arc, we have $h(x)_I=x_I$ for every configuration $x$ on $V_C$. Then,
since $C\setminus I$ is acyclic, for any {\em input configuration}
$z\in\bool^I$, $h$ has exactly one fixed point $x$ extending $z$ (and thus $h$
has exactly $2^{|I|}$ fixed points). We then call $\tilde z= x_O$ the {\em
output configuration computed by $h$} from $z$. Thus, $h$ encodes the map $g$
from $\bool^I$ to $\bool^O$ defined by $g(z)=\tilde z$ or, equivalently, by
$g(x_I)=x_O$ for every fixed point $x$ of $h$. We say that $C$ is a {\em basic
circuit structure} if $\Delta(C)\leq 2$ and every vertex $i$ of in-degree two
has only positive in-neighbors, so that $h_i$ is either the AND gate or the OR
gate. In that case, the fixed points of $h$ can be characterized by a short
3-CNF formula. 

\begin{lemma}\label{lem:omega}
Let $C$ be a basic circuit structure and $h\in\functions(C)$. There is a 3-CNF formula $\omega$, which can be computed in polynomial time with respect to $|V_C|$, whose variables are in $V_C$, and containing at most $3|V_C|$ clauses, such that, for all configurations $x$ on $V_C$:
\[
h(x)=x \quad\iff\quad \textrm{$\omega$ is satisfied by $x$}.
\]
\end{lemma}

\begin{proof}
Since input vertices have a positive loop and no other incoming arc, we only
have to prove that $\omega$ is satisfied by $x$ if and only if $h_i(x)=x_i$
for all non-input vertex $i$. We construct $\omega$ as follows, for every
non-input vertex $i$, and simultaneously prove the lemma.
\begin{itemize}
\item 
If $i$ is a source and $h_i$ is the 1 constant function, we add the clause $\{i^+\}$. This clause is satisfied by $x$ if and only if $x_i=1=h_i(x)$.
\item 
If $i$ is a source and $h_i$ is the 0 constant function, we add the clause $\{i^-\}$. This clause is satisfied by $x$ if and only if $x_i=0=h_i(x)$.
\item 
If $i$ has a unique in-neighbor $j$, which is positive, we add the clauses
$\{i^+,j^-\}$ and $\{i^-,j^+\}$. These clauses are simultaneously
satisfied by $x$ if and only if $x_i=x_j=h_i(x)$.
\item 
If $i$ has a unique in-neighbor $j$, which is negative, we add the clauses
$\{i^+,j^+\}$ and $\{i^-,j^-\}$. These clauses are simultaneously
satisfied by $x$ if and only if $x_i=\neg x_j=h_i(x)$. 
\item
If $i$ has two in-neighbors $j,k$, and $h_i$ is the AND gate, we add three clauses: $\{i^-,j^+\}$, $\{(i^-,k^+\}$ and $\{i^+,j^-,k^-\}$. These clauses are simultaneously satisfied by $x$ if and only if $x_i=x_j\land x_k=h_i(x)$.
\item 
If $i$ has two in-neighbors $j,k$, and $h_i$ is the OR gate, we add three clauses: $\{i^+,j^-\}$, $\{i^+,k^-\}$ and $\{i^-,j^+,k^+\}$. These clauses are simultaneously satisfied by $x$ if and only if $x_i=x_j\lor x_k=h_i(x)$.
\end{itemize}
\end{proof}


Consider a 3-CNF formula $\Psi$ over a set $\GLAMBDA$ of $2^n$ variables and with a set $\M$ of $2^m$ clauses. To simplify notations, we take two sets $W,U$ of size $n$ and $m$ respectively, and~write:
\[
\GLAMBDA=\{\glambda_w\}_{w\in \bool^W}, \quad \M=\{M_u\}_{u\in \bool^U},\qquad M_u=(M_{u,01},M_{u,10},M_{u,11}).
\]
Hence, the clauses of $\Psi$ are regarded, for convenience, as ordered triples of literals (not necessarily all distinct) instead of subsets of literals of size at most $3$.  

\medskip
A {\em succinct representation} of $\Psi$ is then a couple $(h,C)$, where $C$ is a basic circuit structure and $h\in\functions(C)$, with the following specifications: the set of input vertices is $U\cup P$, where $P=\{p_1,p_2\}$, and the set of output vertices is $W\cup\{\rho\}$ (these four sets are pairwise disjoint); and for every fixed point $x$ of $h$ with $x_P\neq 00$, we have:
\begin{equation}\tag{$*$}\label{eq:h}
M_{x_U,x_P}=
\left\{
\begin{array}{ll}
\glambda^+_{x_W}&\textrm{if $x_\rho=1$},\\
\glambda^-_{x_W}&\textrm{if $x_\rho=0$}.
\end{array}
\right.
\end{equation}
Hence, each input configuration with $x_P\neq 00$ corresponds to a clause
($x_U$) and a valid position in that clause ($x_P$), and the output
configuration gives the corresponding positive or negative literal: the
involved variable ($x_W$) and the polarity of the literal ($x_\rho$). See
Figure \ref{fig:h} for an illustration. In all the following, $(h,C)$ is a
given succinct representation of $\Psi$. 

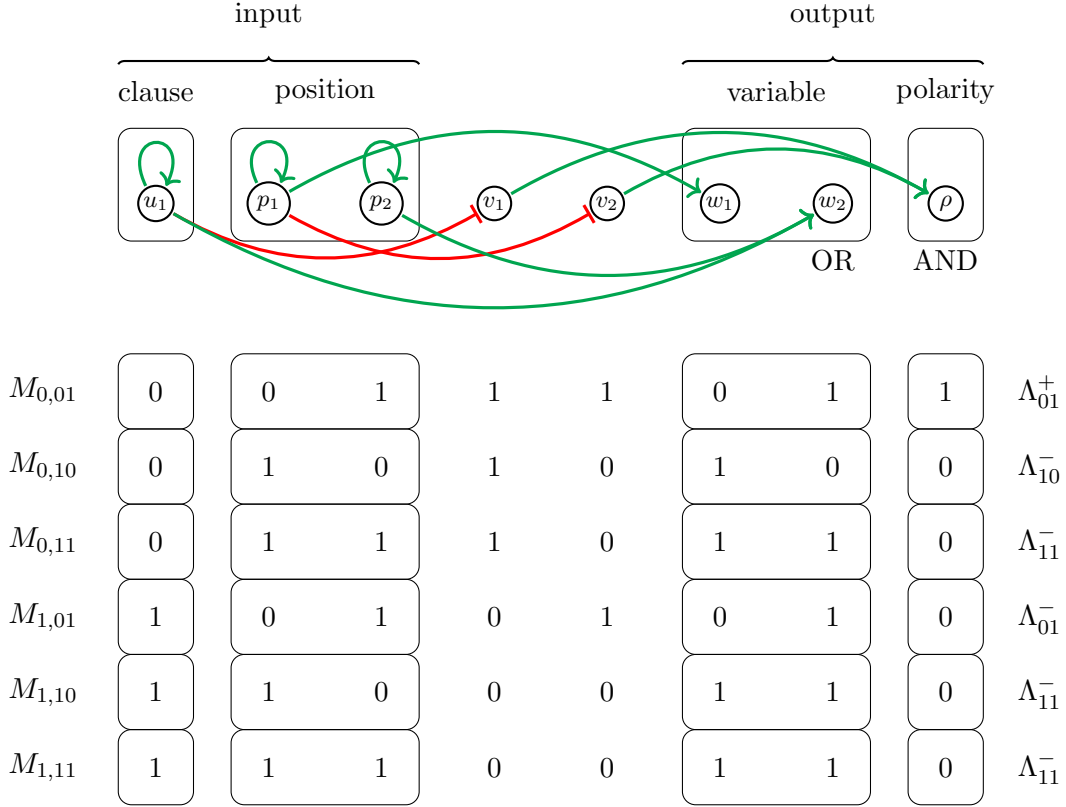
\begin{figure}
\[
\begin{tikzpicture}[scale=1]
\node[outer sep=1,inner sep=1,circle,draw,thick] (u1) at (2,8){\footnotesize$u_1$};
\node[outer sep=1,inner sep=2,circle,draw,thick] (1) at (3.5,8){\footnotesize$p_1$};
\node[outer sep=1,inner sep=2,circle,draw,thick] (2) at (5,8){\footnotesize$p_2$};
\node[outer sep=1,inner sep=1,circle,draw,thick] (v1) at (6.5,8){\footnotesize$v_1$};
\node[outer sep=1,inner sep=1,circle,draw,thick] (v2) at (8,8){\footnotesize$v_2$};
\node[outer sep=1,inner sep=1,circle,draw,thick] (w1) at (9.5,8){\footnotesize$w_1$};
\node[outer sep=1,inner sep=1,circle,draw,thick] (w2) at (11,8){\footnotesize$w_2$};
\node[outer sep=1,inner sep=2,circle,draw,thick] (rho) at (12.5,8){\footnotesize$\rho$};
\node at (11,7.25){\textrm{OR}};
\node at (12.5,7.25){\textrm{AND}};
\draw[rounded corners=5] (1.5,7.5) rectangle (2.5,9);
\node at (2,9.5){clause};
\draw[rounded corners=5] (3,7.5) rectangle (5.5,9);
\node at (4.25,9.5){position};
\draw[decoration={brace,mirror,raise=10},decorate,thick] (5.5,9.5) --  (1.5,9.5);
\node at (3.5,10.5){input};
\draw[decoration={brace,mirror,raise=10},decorate,thick] (13,9.5) --  (9,9.5);
\node at (11,10.5){output};
\draw[rounded corners=5] (9,7.5) rectangle (11.5,9);
\node at (10.25,9.5){variable};
\draw[rounded corners=5] (12,7.5) rectangle (13,9);
\node at (12.5,9.5){polarity};
\draw[Green,->,very thick] (u1.120) .. controls (1.5,9) and (2.5,9) .. (u1.60);
\draw[Green,->,very thick] (1.120) .. controls (3,9) and (4,9) .. (1.60);
\draw[Green,->,very thick] (2.120) .. controls (4.5,9) and (5.5,9) .. (2.60);
\path[Green,->,very thick]
(u1) edge[red,bend right=30,-|] (v1)
(1) edge[red,bend right=30,-|] (v2)
(1) edge[bend left=30] (w1)
(u1) edge[bend right=30] (w2)
(2) edge[bend right=30] (w2)
(v1) edge[bend left=30] (rho)
(v2) edge[bend left=30] (rho)
;
\node at (0.5,5.5){$M_{0,01}$};
\node at   (2,5.5){0};
\node at (3.5,5.5){0};
\node at   (5,5.5){1};
\node at (6.5,5.5){1};
\node at   (8,5.5){1};
\node at (9.5,5.5){0};
\node at  (11,5.5){1};
\node at(12.5,5.5){1};
\node at (13.75,5.5){$\Lambda^+_{01}$};
\draw[rounded corners=5](1.5,5) rectangle (2.5,6);
\draw[rounded corners=5]  (3,5) rectangle (5.5,6);
\draw[rounded corners=5]  (9,5) rectangle(11.5,6);
\draw[rounded corners=5] (12,5) rectangle  (13,6);
\node at (0.5,4.5){$M_{0,10}$};
\node at   (2,4.5){0};
\node at (3.5,4.5){1};
\node at   (5,4.5){0};
\node at (6.5,4.5){1};
\node at   (8,4.5){0};
\node at (9.5,4.5){1};
\node at  (11,4.5){0};
\node at(12.5,4.5){0};
\node at (13.75,4.5){$\Lambda^-_{10}$};
\draw[rounded corners=5](1.5,4) rectangle (2.5,5);
\draw[rounded corners=5]  (3,4) rectangle (5.5,5);
\draw[rounded corners=5]  (9,4) rectangle(11.5,5);
\draw[rounded corners=5] (12,4) rectangle  (13,5);
\node at (0.5,3.5){$M_{0,11}$};
\node at   (2,3.5){0};
\node at (3.5,3.5){1};
\node at   (5,3.5){1};
\node at (6.5,3.5){1};
\node at   (8,3.5){0};
\node at (9.5,3.5){1};
\node at  (11,3.5){1};
\node at(12.5,3.5){0};
\node at (13.75,3.5){$\Lambda^-_{11}$};
\draw[rounded corners=5](1.5,3) rectangle (2.5,4);
\draw[rounded corners=5]  (3,3) rectangle (5.5,4);
\draw[rounded corners=5]  (9,3) rectangle(11.5,4);
\draw[rounded corners=5] (12,3) rectangle  (13,4);
\node at (0.5,2.5){$M_{1,01}$};
\node at   (2,2.5){1};
\node at (3.5,2.5){0};
\node at   (5,2.5){1};
\node at (6.5,2.5){0};
\node at   (8,2.5){1};
\node at (9.5,2.5){0};
\node at  (11,2.5){1};
\node at(12.5,2.5){0};
\node at (13.75,2.5){$\Lambda^-_{01}$};
\draw[rounded corners=5](1.5,2) rectangle (2.5,3);
\draw[rounded corners=5]  (3,2) rectangle (5.5,3);
\draw[rounded corners=5]  (9,2) rectangle(11.5,3);
\draw[rounded corners=5] (12,2) rectangle  (13,3);
\node at (0.5,1.5){$M_{1,10}$};
\node at   (2,1.5){1};
\node at (3.5,1.5){1};
\node at   (5,1.5){0};
\node at (6.5,1.5){0};
\node at   (8,1.5){0};
\node at (9.5,1.5){1};
\node at  (11,1.5){1};
\node at(12.5,1.5){0};
\node at (13.75,1.5){$\Lambda^-_{11}$};
\draw[rounded corners=5](1.5,1) rectangle (2.5,2);
\draw[rounded corners=5]  (3,1) rectangle (5.5,2);
\draw[rounded corners=5]  (9,1) rectangle(11.5,2);
\draw[rounded corners=5] (12,1) rectangle  (13,2);
\node at (0.5,0.5){$M_{1,11}$};
\node at   (2,0.5){1};
\node at (3.5,0.5){1};
\node at   (5,0.5){1};
\node at (6.5,0.5){0};
\node at   (8,0.5){0};
\node at (9.5,0.5){1};
\node at  (11,0.5){1};
\node at(12.5,0.5){0};
\node at (13.75,0.5){$\Lambda^-_{11}$};
\draw[rounded corners=5](1.5,0) rectangle (2.5,1);
\draw[rounded corners=5]  (3,0) rectangle (5.5,1);
\draw[rounded corners=5]  (9,0) rectangle(11.5,1);
\draw[rounded corners=5] (12,0) rectangle  (13,1);

\end{tikzpicture}
\]
{\caption{\label{fig:h}
A circuit $h\in F(C)$ encoding a 3-CNF formula $\Psi$ over the set of variables $\GLAMBDA=\{\Lambda_{00},\Lambda_{01},\Lambda_{10},\Lambda_{11}\}$ and containing the clauses $M_0=(M_{0,01},M_{0,10},M_{0,11})=(\Lambda^+_{01},\Lambda^-_{10},\Lambda^-_{11})$ and $M_1=(M_{1,01},M_{1,10},M_{1,11})=(\Lambda^-_{01},\Lambda^-_{11},\Lambda^-_{11})$. In other words $\Psi=(\Lambda_{01}\lor\neg\Lambda_{10}\lor\neg \Lambda_{11})\land(\neg \Lambda_{01}\lor\neg \Lambda_{11})$; this formula is equivalent to the formula $\psi$ of Figure~\ref{fig:D_psi}. The circuit structure $C$, which is basic, is drawn on the top; the set of input vertices is $U\cup P$ with $U=\{u_1\}$ and $P=\{p_1,p_2\}$, and the set of output vertices is $W\cup\{\rho\}$ with $W=\{w_1,w_2\}$. The circuit $h$ is the BN in $\functions(C)$ such that $h_{w_2}$ is the OR function and $h_\rho$ is the AND function (there is a unique possible local function for the other vertices). The 6 fixed points $x$ of $h$ with $x_P\neq 00$ are displayed. They encode the formula $\Psi$ as indicated. 
}}
\end{figure}

\medskip
The next definition provides some flexibility regarding the representation of $\Psi$. 

\begin{definition}[$\epsilon$-fixed points and consistency]
Let $C,C'$ be circuit structures with $V_C\subseteq V_{C'}$,
$h'\in\functions(C')$ and $\epsilon$ be a configuration on $V_{C'}$. A
configuration $x$ on $V_{C'}$ is an {\em $\epsilon$-fixed point} of $h'$ if
$x\xor\epsilon$ is a fixed point of $h'$. We say that $h'$ is {\em
$\epsilon$-consistent} if every $\epsilon$-fixed point $x$ of $h'$ extends a
fixed point of $h$ and has a valid position, that is, $x_P\neq 00$ (thus
$(*)$ holds for every $\epsilon$-fixed point $x$, so $h'$ embeds in some way
the calculations performed~by~$h$). We say that $h'$ is {\em consistent} if
it is $\epsilon$-consistent with $\epsilon=0$.
\end{definition}

We will construct, from the succinct representation $(h,C)$ of $\Psi$, \an{} SID
$D_{\Psi}$ with a number of vertices linear in $|V_C|$ such that
$\maxpf(D_\Psi)\geq 2^{m+1}$ if and only if $\Psi$ is satisfiable; and since
$D_\psi$ will have a positive feedback vertex set of size $m+1$, the inequality
in this equivalence is actually an equality (by Theorem~\ref{thm:bound}).
We first give a sketch of the construction, proceeding in three steps. 

\medskip
First, we construct a circuit structure $C'$ over $C$, whose input
configurations correspond to the clauses of $\Psi$ (so the corresponding BNs
have $2^m$ fixed points, one per clause) and with two output vertices: $\rho$
(which is already in $C$) and $\nu$. We then show two properties.
\begin{enumerate}
\item[$(1)$]
Given an assignment $\zeta\in\bool^{\GLAMBDA}$, there is a consistent BN
$h'\in\functions(C')$ such that, for every fixed point $x$, the clause
$M_{x_U}$ is satisfied by $\zeta$ if and only if $x_\rho=x_\nu$.
\item[$(2)$]
Conversely, for every $\epsilon$-consistent BN $h'\in\functions(C')$ there is
an assignment $\zeta\in\bool^{\GLAMBDA}$ such that, for every
$\epsilon$-fixed point $x$, the clause $M_{x_U}$ is satisfied by $\zeta$
whenever $x_\rho=x_\nu$.
\end{enumerate}

\medskip
The trick is then to encode the consistency conditions and the condition
``$x_\rho=x_\nu$'' into a 3-CNF formula. First consider the formula $\omega$
characterizing the fixed points of $h$. Then configuration $x$ on $V_{C'}$
extends a fixed point of $h$ if and only if $\omega$ is satisfied by $x_{V_C}$.
By adding few clauses in $\omega$, we can then obtain a formula $\psi$ which is
satisfied by $x$ if and only if $x$ extends a fixed point of $h$, $x_P\neq 00$
and $x_\rho=x_\nu$. In particular, if $\psi$ is satisfied by the $2^m$
$\epsilon$-fixed points of $h'$, then $h'$ is $\epsilon$-consistent. Now, if
$\Psi$ is satisfied by $\zeta$, then $\psi$ is satisfied by the $2^m$ fixed
points of the consistent BN $h'$ described in $(1)$. Conversely, for any $h'\in
\functions(C')$ and $\epsilon$, if $\psi$ is satisfied by the $2^m$
$\epsilon$-fixed points of $h'$, then $\Psi$ is satisfied by the assignment
$\zeta$ described in $(2)$. Hence, we obtain:
\begin{enumerate}
\item[$(3)$]
The following conditions are equivalent: (a) $\Psi$ is satisfiable; (b) there is $h'\in\functions(C')$ such that $\psi$ is satisfied by the $2^m$ fixed points of $h'$; (c) there are $h'\in\functions(C')$ and $\epsilon$ such that $\psi$ is satisfied by the $2^m$ $\epsilon$-fixed points of $h'$.
\end{enumerate}

\medskip
The last step is to consider $D_\psi$. Actually, we define $D_{\Psi}$ as the
extension of $D_\psi$ resulting from the union of $C'$ and $D_\psi$. If $\Psi$
is satisfiable, there is a BN $f\in\functions(D_\Psi)$, satisfying the
conditions of   Lemma~\ref{lem:ext1}, whose partial fixed points are the fixed
points of the BN $h'\in\functions(C')$ described in (b). Then, since the $2^m$
partial fixed points are satisfying assignments of $\psi$, they give rise to
$2^{m+1}$ fixed points for $f$. Conversely, for any $f\in\functions(D_\Psi)$,
the partial fixed points of $f$ are the fixed points of some
$h'\in\functions(C')$, and if $f$ has $2^{m+1}$ fixed points, then each fixed
point $x$ of $h'$ gives rise to two fixed points of $f$. We then deduce from
Lemma~\ref{lem:ext2} that this is possible only if there is $\epsilon$ such
that, for every fixed point $x$ of $h'$, $\psi$ is satisfied by
$x\xor\epsilon$. This is equivalent to say that $\psi$ is satisfied by the
$2^m$ $\epsilon$-fixed points of $h'$, so (c) holds and we deduce that $\Psi$
is satisfiable. 

\medskip
We now proceed to the details. 

\begin{definition}[$C'$]\label{def:C'}
We denote by $C'$ the circuit structure obtained from $C$ by removing the positive loops on vertices $p_1,p_2$, and by adding three new vertices, $s_1,s_2,\nu$, and the following arcs:
	\begin{itemize}
	\item a \Null{} arc $(j,p_1)$ and a \Null{} arc $(j,p_2)$, for all $j\in U\cup\{s_1,s_2\}$; 
	\item a \Null{} arc $(j,\nu)$, for all $j\in W\cup\{s_1,s_2\}$. 
	\end{itemize}
\end{definition}

So $C'$ is a circuit structure where $U$ is the set of input vertices (input configurations correspond to clauses of $\Psi$)  and $\{\rho,\nu\}$ is the set of output vertices. The vertices $s_1,s_2$ are sources. See the top of Figure~\ref{fig:D_Psi} for an illustration. 

\medskip
The following lemma is a formal statement of the property $(1)$ discussed above.
	
\begin{lemma}\label{lem:pro1}
For every assignment $\zeta\in\bool^{\GLAMBDA}$ there is a consistent circuit $h'\in \functions(C')$ such that, for every fixed point $x$ of $h'$, 
\[
\textrm{$M_{x_U}$ is satisfied by $\zeta$}
\quad\iff\quad
x_\rho=x_\nu. 
\]
\end{lemma}

\begin{proof}
For every configuration $x$ on $V_{C'}$, we define $h'(x)$ componentwise as follows. Since every non-input vertex $i$ of $C$ has exactly the same incoming arcs in $C'$ and $C$, we can set $h'_i(x)=h_i(x_{V_C})$. Next we define $h'_{s_1}$ and $h'_{s_2}$ as the 0 constant function. Since each $i\in U$ has a positive loop and nothing else, we necessarily set $h'_i(x)=x_i$. It remains to define the local functions of $p_1,p_2$ and $\nu$. Let $S=\{s_1,s_2\}$. 

\medskip
First, we set  
\[
\begin{array}{l}
h'_{p_1}(x)=
\left\{
\begin{array}{ll}
\bigoplus_{i \in U\cup S} x_i & \textrm{if $x_S\neq 00$},\\
0&\textrm{if $x_S=00$ and $M_{x_U,01}$ is satisfied by $\zeta$},\\
1&\textrm{otherwise.}
\end{array}
\right.
\\
~
\\
h'_{p_2}(x)=
\left\{
\begin{array}{ll}
\bigoplus_{i \in U\cup S} x_i & \textrm{if $x_S\neq 00$},\\
0&\textrm{if $x_S=00$ and $M_{x_U,10}$ is satisfied by $\zeta$ but not $M_{x_U,01}$},\\
1&\textrm{otherwise.}
\end{array}
\right.
\end{array}
\]
It is clear $h'_{p_1}$ and $h'_{p_2}$ only depend on components in $U\cup S$, and the first case in each definition ensures that these dependencies are effective and neither positive nor negative, so that $h'_{p_1}\in \functions_{p_1}(C')$ and $h'_{p_2}\in\functions_{p_2}(C')$. Furthermore, if $h'(x)=x$ then $x_S=00$, and it is easy to check that this implies $x_P\neq 00$ and the following equivalence:
\[
\text{$M_{x_U}$ is satisfied by $\zeta$} \quad\iff\quad \text{$M_{x_U,x_P}$ is satisfied by $\zeta$}.  
\]

\medskip
Second, we set 
\[
h'_\nu(x)=
\left\{
\begin{array}{ll}
\bigoplus_{i \in W\cup S} x_i & \textrm{if $x_S\neq 00$}\\
\zeta_{\glambda_{x_W}}&\textrm{otherwise.}
\end{array}
\right.
\]
Similarly, $h'_{\nu}$ only depends on components in $W\cup S$, and the first case ensures that these dependencies are effective and neither positive nor negative, so that $h'_{\nu}\in \functions_{\nu}(C')$. 

\medskip
Suppose that $h'(x)=x$. As said above, we have $x_S=00$ and $x_P\neq 00$. Furthermore, for each non-input vertex $i$ of $C$ we have $h_i(x_{V_C})=h'_i(x)=x_i$, thus $x_{V_C}$ is a fixed point of $h$. Hence, $h'$ is consistent, and $M_{x_U,x_P}$ is $\glambda^+_{x_W}$ if $x_\rho=1$ and $\glambda^-_{x_W}$ otherwise. Since $x_\nu=h'_\nu(x)=\zeta_{\glambda_{x_{W}}}$, the literal $M_{x_U,x_P}$ is satisfied by $\zeta$ if and only if $x_\rho=x_\nu$, and by the equivalence above, the clause $M_{x_U}$ is satisfied by $\zeta$ if and only if $x_\rho=x_\nu$.
\end{proof}	

We now prove a kind of converse, which is a formal statement of the property $(2)$ discussed above. 

\begin{lemma}\label{lem:pro2}
For every $\epsilon$-consistent circuit $h'\in \functions(C')$ there is an assignment $\zeta\in\bool^{\GLAMBDA}$ such that, for every $\epsilon$-fixed point $x$ of $h'$, 
\[
x_\rho=x_\nu
\quad\Longrightarrow\quad
\textrm{$M_{x_U}$ is satisfied by $\zeta$}.
\] 
\end{lemma}

\begin{proof}
Since $s_1,s_2$ are sources, for every fixed points $x,y$ of $h'$ we have $x_{s_1}=y_{s_1}$ and $x_{s_2}=y_{s_2}$. Thus, if $x_W=y_W$ then $x_\nu=h'_\nu(x)=h'_\nu(y)=y_\nu$, because all the in-neighbors of $\nu$ are in $W\cup \{s_1,s_2\}$. In other words, there is a function $g$ from $\bool^W$ to $\bool$ such that $g(x_W)=x_\nu$ for every fixed point $x$ of $h'$. Let any $\zeta\in\bool^{\GLAMBDA}$ satisfying 
\[
\zeta_{\glambda_{x_W\xor\epsilon_W}}=g(x_W)\xor\epsilon_\nu=x_\nu\xor\epsilon_\nu
\]
for every fixed point $x$ of $h'$. This is equivalent to say that, for every $\epsilon$-fixed point $x$ of $h'$, 
\[
\zeta_{\glambda_{x_W}}=x_\nu.
\]
Now, since $h'$ is $\epsilon$-consistent, for every $\epsilon$-fixed point $x$ of $h'$ we have $x_P\neq 00$ and $M_{x_U,x_P}$ is $\glambda^+_{x_W}$ if $x_\rho=1$ and $\glambda^-_{x_W}$ otherwise. If $x_\rho=x_\nu=\zeta_{\glambda_{x_W}}$, then this literal $M_{x_U,x_P}$ is satisfied by $\zeta$, and so is the clause $M_{x_U}$. 
\end{proof}

Here is the definition of the formula $\psi$ encoding the consistency conditions and the condition  ``$x_\rho=x_\nu$''.

\begin{definition}[Formula $\psi$ associated with $h$]\label{def:psi_h}
  Let $\omega$ be the 3-CNF formula characterizing the fixed points of $h$, as
  in Lemma~\ref{lem:omega}. The \emph{formula $\psi$ associated with $h$} is
  the 3-CNF formula obtained from $\omega$ by adding the following three
  clauses: $\{p_1^+,p_2^+\}$, $\{\rho^+,\nu^-\}$ and
  $\{\rho^-,\nu^+\}$. So $\psi$ is a formula over the set of variables
  $\LAMBDA=V_C\cup\{\nu\}$.
\end{definition}

Note that, given a fixed point $x$ of $h'$, $\omega$ is satisfied by $x$ if and only if $x$ extends a fixed point of $h$ (since $\omega$ is satisfied by $x$ if and only if $\omega$ is satisfied by $x_{V_C}$ if and only if $x_{V_C}$ is a fixed point of $h$), and the three additional clauses are simultaneously satisfied by $x$ if and only if $x_P\neq 00$ and $x_\rho=x_\nu$. In particular, if $\psi$ is satisfied by every $\epsilon$-fixed point of $h'$, then $h'$ is $\epsilon$-consistent. 
 
\medskip
We now prove a formal statement of the property $(3)$ discussed above.  

\begin{lemma}\label{lem:pro3}
The following conditions are equivalent:
\begin{itemize}
\item[\em (a)]
$\Psi$ is satisfiable;
\item[\em (b)] there is $h'\in\functions(C')$ such that $\psi$ is satisfied by every fixed point of $h'$;
\item[\em (c)] there are $h'\in\functions(C')$ and $\epsilon\in\bool^{V_{C'}}$ such that $\psi$ is satisfied by every $\epsilon$-fixed point of $h'$.
\end{itemize}
\end{lemma}

\begin{proof}
Suppose that $\Psi$ is satisfied by $\zeta\in\bool^{\GLAMBDA}$. By Lemma~\ref{lem:pro1}, there is a BN $h'\in \functions(C')$ consistent with $h$ such that $x_\rho=x_\nu$ for every fixed point $x$ of $h$. By consistency, $x_P\neq 00$ and $x$ extends a fixed point of $h$, thus $\omega$ is satisfied. We deduce that $\psi$ is satisfied by every fixed point of $h'$. This proves that (a) implies (b). Since (b) trivially implies (c), it remains to prove that (c) implies (a). 

\medskip
Suppose $h'$ and $\epsilon$ are as in (c). Then $h'$ is $\epsilon$-consistent, and $x_\rho=x_\nu$ for every $\epsilon$-fixed point $x$ of $h'$. Hence, by Lemma~\ref{lem:pro2}, there is an assignment $\zeta\in\bool^{\GLAMBDA}$ such that, for every $\epsilon$-fixed point $x$ of $h'$, $M_{x_U}$ is satisfied by $\zeta$. Since $U$ is the set of input vertices of~$C'$, any input configuration, that is any member of in $\bool^{U}$, is extended by a fixed point of $h'$. Hence, for any input configuration $u$, the input configuration  $u\xor\epsilon_U$ is extended by a fixed point $x$ of $h'$, so $u$ is extended by $x\xor\epsilon$, which is an $\epsilon$-fixed point of $h'$. Hence, all the clauses of $\Psi$ are satisfied by $\zeta$. This proves that (c) implies (a). 
\end{proof}

The formal definition of our construction $D_\Psi$ follows, see Figure \ref{fig:D_Psi} for an illustration. 

\begin{definition}[$D_\Psi$]\label{def:D_Psi}
  Let $C'$ be the circuit structure of Definition~\ref{def:C'}. Let $\psi$ be the
  3-CNF formula associated with $h$ given in Definition~\ref{def:psi_h}, which
  is a formula over the set of variables $\LAMBDA=V_C\cup\{\nu\}$. Let $D_\psi$
  be the SID defined from $\psi$ as in Definition~\ref{def:D_psi}. We define
  $D_\Psi$ as the extension of $D_\psi$ obtained by taking the union of $C'$
  and $D_\psi$ (supposing naturally that vertices $s_1,s_2$ are not in
  $D_\psi\setminus\LAMBDA$). We denote by $V_\Psi$ the vertex set of~$D_\Psi$. 
\end{definition}

Since $\psi$ has at most $3|V_C|+3$ clauses and $|\LAMBDA|=|V_C|+1$, we have 
\[
|V_\Psi|=|V_\psi|+2\leq \big(4|\LAMBDA|+2(3|V_C|+3)+1\big)+2=10|V_C|+13,
\]
which is linear according to the size of the succinct representation of $\Psi$. 

\begin{figure}
\vspace{-1.5cm}
\[
\hspace{-1.5cm}
\begin{tikzpicture}[scale=0.7]
\node[outer sep=1,inner sep=1,circle,draw,thick] (t1) at (7,14){\footnotesize$s_1$};
\node[outer sep=1,inner sep=1,circle,draw,thick] (t2) at (9,14){\footnotesize$s_2$};
\node[outer sep=1,inner sep=1,circle,draw,thick] (u1) at (1,8){\footnotesize$u_1$};
\node[outer sep=1,inner sep=1,circle,draw,thick] (1) at (3,8){\footnotesize$p_1$};
\node[outer sep=1,inner sep=1,circle,draw,thick] (2) at (5,8){\footnotesize$p_2$};
\node[outer sep=1,inner sep=1,circle,draw,thick] (v1) at (7,8){\footnotesize$v_1$};
\node[outer sep=1,inner sep=1,circle,draw,thick] (v2) at (9,8){\footnotesize$v_2$};
\node[outer sep=1,inner sep=1,circle,draw,thick] (w1) at (11,8){\footnotesize$w_1$};
\node[outer sep=1,inner sep=1,circle,draw,thick] (w2) at (13,8){\footnotesize$w_2$};
\node[outer sep=1,inner sep=2.5,circle,draw,thick] (rho) at (15,8){\footnotesize$\rho$};
\node[outer sep=1,inner sep=2.5,circle,draw,thick] (nu) at (17,8){\footnotesize$\nu$};
\node[outer sep=1,inner sep=0.5,circle,draw,thick] (u1+) at (1,3){\footnotesize$u^+_1$};
\node[outer sep=1,inner sep=0.5,circle,draw,thick] (1+) at (3,3){\footnotesize$p_1^+$};
\node[outer sep=1,inner sep=0.5,circle,draw,thick] (2+) at (5,3){\footnotesize$p_2^+$};
\node[outer sep=1,inner sep=0.5,circle,draw,thick] (v1+) at (7,3){\footnotesize$v^+_1$};
\node[outer sep=1,inner sep=0.5,circle,draw,thick] (v2+) at (9,3){\footnotesize$v^+_2$};
\node[outer sep=1,inner sep=0.5,circle,draw,thick] (w1+) at (11,3){\footnotesize$w^+_1$};
\node[outer sep=1,inner sep=0.5,circle,draw,thick] (w2+) at (13,3){\footnotesize$w^+_2$};
\node[outer sep=1,inner sep=0.5,circle,draw,thick] (rho+) at (15,3){\footnotesize$\rho^+$};
\node[outer sep=1,inner sep=0.5,circle,draw,thick] (nu+) at (17,3){\footnotesize$\nu^+$};
\node[outer sep=1,inner sep=0.5,circle,draw,thick] (u1-) at (1,1){\footnotesize$u^-_1$};
\node[outer sep=1,inner sep=0.5,circle,draw,thick] (1-) at (3,1){\footnotesize$p_1^-$};
\node[outer sep=1,inner sep=0.5,circle,draw,thick] (2-) at (5,1){\footnotesize$p_2^-$};
\node[outer sep=1,inner sep=0.5,circle,draw,thick] (v1-) at (7,1){\footnotesize$v^-_1$};
\node[outer sep=1,inner sep=0.5,circle,draw,thick] (v2-) at (9,1){\footnotesize$v^-_2$};
\node[outer sep=1,inner sep=0.5,circle,draw,thick] (w1-) at (11,1){\footnotesize$w^-_1$};
\node[outer sep=1,inner sep=0.5,circle,draw,thick] (w2-) at (13,1){\footnotesize$w^-_2$};
\node[outer sep=1,inner sep=0.5,circle,draw,thick] (rho-) at (15,1){\footnotesize$\rho^-$};
\node[outer sep=1,inner sep=0.5,circle,draw,thick] (nu-) at (17,1){\footnotesize$\nu^-$};
\node[outer sep=1,inner sep=2,circle,draw,thick] (muv11) at (0,-4){\footnotesize$$};
\node[outer sep=1,inner sep=2,circle,draw,thick] (muv12) at (1.5,-4){\footnotesize$$};
\node[outer sep=1,inner sep=2,circle,draw,thick] (muv21) at (3,-4){\footnotesize$$};
\node[outer sep=1,inner sep=2,circle,draw,thick] (muv22) at (4.5,-4){\footnotesize$$};
\node[outer sep=1,inner sep=2,circle,draw,thick] (muw11) at (6,-4){\footnotesize$$};
\node[outer sep=1,inner sep=2,circle,draw,thick] (muw12) at (7.5,-4){\footnotesize$$};
\node[outer sep=1,inner sep=2,circle,draw,thick] (muw21) at (9,-4){\footnotesize$$};
\node[outer sep=1,inner sep=2,circle,draw,thick] (muw22) at (10.5,-4){\footnotesize$$};
\node[outer sep=1,inner sep=2,circle,draw,thick] (muw23) at (12,-4){\footnotesize$$};
\node[outer sep=1,inner sep=2,circle,draw,thick] (murho1) at (13.5,-4){\footnotesize$$};
\node[outer sep=1,inner sep=2,circle,draw,thick] (murho2) at (15,-4){\footnotesize$$};
\node[outer sep=1,inner sep=2,circle,draw,thick] (murho3) at (16.5,-4){\footnotesize$$};
\node[outer sep=1,inner sep=2,circle,draw,thick] (mu12) at (-1.5,-4){\footnotesize$$};
\node[outer sep=1,inner sep=2,circle,draw,thick] (murhonu1) at (18,-4){\footnotesize$$};
\node[outer sep=1,inner sep=2,circle,draw,thick] (murhonu2) at (19.5,-4){\footnotesize$$};
\node[outer sep=1,inner sep=2,circle,draw,thick] (Cmuv11) at (0,-5){\footnotesize$$};
\node[outer sep=1,inner sep=2,circle,draw,thick] (Cmuv12) at (1.5,-5){\footnotesize$$};
\node[outer sep=1,inner sep=2,circle,draw,thick] (Cmuv21) at (3,-5){\footnotesize$$};
\node[outer sep=1,inner sep=2,circle,draw,thick] (Cmuv22) at (4.5,-5){\footnotesize$$};
\node[outer sep=1,inner sep=2,circle,draw,thick] (Cmuw11) at (6,-5){\footnotesize$$};
\node[outer sep=1,inner sep=2,circle,draw,thick] (Cmuw12) at (7.5,-5){\footnotesize$$};
\node[outer sep=1,inner sep=2,circle,draw,thick] (Cmuw21) at (9,-5){\footnotesize$$};
\node[outer sep=1,inner sep=2,circle,draw,thick] (Cmuw22) at (10.5,-5){\footnotesize$$};
\node[outer sep=1,inner sep=2,circle,draw,thick] (Cmuw23) at (12,-5){\footnotesize$$};
\node[outer sep=1,inner sep=2,circle,draw,thick] (Cmurho1) at (13.5,-5){\footnotesize$$};
\node[outer sep=1,inner sep=2,circle,draw,thick] (Cmurho2) at (15,-5){\footnotesize$$};
\node[outer sep=1,inner sep=2,circle,draw,thick] (Cmurho3) at (16.5,-5){\footnotesize$$};
\node[outer sep=1,inner sep=2,circle,draw,thick] (Cmu12) at (-1.5,-5){\footnotesize$$};
\node[outer sep=1,inner sep=2,circle,draw,thick] (Cmurhonu1) at (18,-5){\footnotesize$$};
\node[outer sep=1,inner sep=2,circle,draw,thick] (Cmurhonu2) at (19.5,-5){\footnotesize$$};
\node[outer sep=1,inner sep=1,circle,draw,thick] (l0) at (0,2){\footnotesize$$};
\node[outer sep=1,inner sep=1,circle,draw,thick] (l1) at (2,2){\footnotesize$$};
\node[outer sep=1,inner sep=1,circle,draw,thick] (l2) at (4,2){\footnotesize$$};
\node[outer sep=1,inner sep=1,circle,draw,thick] (l3) at (6,2){\footnotesize$$};
\node[outer sep=1,inner sep=1,circle,draw,thick] (l4) at (8,2){\footnotesize$$};
\node[outer sep=1,inner sep=1,circle,draw,thick] (l5) at (10,2){\footnotesize$$};
\node[outer sep=1,inner sep=1,circle,draw,thick] (l6) at (12,2){\footnotesize$$};
\node[outer sep=1,inner sep=1,circle,draw,thick] (l7) at (14,2){\footnotesize$$};
\node[outer sep=1,inner sep=1,circle,draw,thick] (l8) at (16,2){\footnotesize$$};
\node[outer sep=1,inner sep=1,circle,draw,thick] (l9) at (18,2){\footnotesize$$};
\path[Green,->,thick]
(1+) edge (mu12)
(2+) edge[bend left=10] (mu12)
(v1+) edge[bend left=15] (muv11)
(u1+) edge[bend right=10] (muv11)
(v1-) edge (muv12)
(u1-) edge (muv12)
(v2+) edge[bend left=10] (muv21)
(1+) edge[bend right=18] (muv21)
(v2-) edge (muv22)
(1-) edge (muv22)
(w1+) edge[bend left=5] (muw11)
(1-) edge (muw11)
(w1-) edge (muw12)
(1+) edge (muw12)
(w2+) edge (muw21)
(u1-) edge (muw21)
(w2+) edge (muw22)
(2-) edge (muw22)
(w2-) edge (muw23)
(u1+) edge[bend right=2] (muw23)
(2+) edge[bend right=15] (muw23)
(rho-) edge (murho1)
(v1+) edge[bend right=10] (murho1)
(rho-) edge (murho2)
(v2+) edge[bend right=10] (murho2)
(rho+) edge[bend left=5] (murho3)
(v1-) edge (murho3)
(v2-) edge (murho3)
(rho+) edge (murhonu1)
(nu-) edge (murhonu1)
(rho-) edge (murhonu2)
(nu+) edge (murhonu2)
;
\path[Green,->,very thick]
(l0) edge (u1+)
(l0) edge (u1-)
(u1+) edge (l1)
(u1-) edge (l1)
(l1) edge (1+)
(l1) edge (1-)
(1+) edge (l2)
(1-) edge (l2)
(l2) edge (2+)
(l2) edge (2-)
(2+) edge (l3)
(2-) edge (l3)
(l3) edge (v1+)
(l3) edge (v1-)
(v1+) edge (l4)
(v1-) edge (l4)
(l4) edge (v2+)
(l4) edge (v2-)
(v2+) edge (l5)
(v2-) edge (l5)
(l5) edge (w1+)
(l5) edge (w1-)
(w1+) edge (l6)
(w1-) edge (l6)
(l6) edge (w2+)
(l6) edge (w2-)
(w2+) edge (l7)
(w2-) edge (l7)
(l7) edge (rho+)
(l7) edge (rho-)
(rho+) edge (l8)
(rho-) edge (l8)
(l8) edge (nu+)
(l8) edge (nu-)
(nu+) edge (l9)
(nu-) edge (l9)
(l9) edge[bend left=70] (Cmurhonu2)
(Cmu12) edge[bend left=70] (l0)
(Cmurhonu2) edge (Cmurhonu1)
(Cmurhonu1) edge (Cmurho3)
(Cmurho3) edge (Cmurho2)
(Cmurho2) edge (Cmurho1)
(Cmurho1) edge (Cmuw23)
(Cmuw23) edge (Cmuw22)
(Cmuw22) edge (Cmuw21)
(Cmuw21) edge (Cmuw12)
(Cmuw12) edge (Cmuw11)
(Cmuw11) edge (Cmuv22)
(Cmuv22) edge (Cmuv21)
(Cmuv21) edge (Cmuv12)
(Cmuv12) edge (Cmuv11)
(Cmuv11) edge (Cmu12)
(murhonu2) edge[red,-|] (Cmurhonu2)
(murhonu1) edge[red,-|] (Cmurhonu1)
(murho3) edge[red,-|] (Cmurho3)
(murho2) edge[red,-|] (Cmurho2)
(murho1) edge[red,-|] (Cmurho1)
(muw23) edge[red,-|] (Cmuw23)
(muw22) edge[red,-|] (Cmuw22)
(muw21) edge[red,-|] (Cmuw21)
(muw12) edge[red,-|] (Cmuw12)
(muw11) edge[red,-|] (Cmuw11)
(muv22) edge[red,-|] (Cmuv22)
(muv21) edge[red,-|] (Cmuv21)
(muv12) edge[red,-|] (Cmuv12)
(muv11) edge[red,-|] (Cmuv11)
(mu12) edge[red,-|] (Cmu12)
;
\path[Green,->,thick]
(u1) edge (u1+)
(1) edge (1+)
(2) edge (2+)
(v1) edge (v1+)
(v2) edge (v2+)
(w1) edge (w1+)
(w2) edge (w2+)
(rho) edge (rho+)
(nu) edge (nu+)
;
\path[red,-|,thick,bend left=18]
(u1) edge (u1-)
(1) edge (1-)
(2) edge (2-)
(v1) edge (v1-)
(v2) edge (v2-)
(w1) edge (w1-)
(w2) edge (w2-)
(rho) edge (rho-)
(nu) edge (nu-)
;
\path[->,thick]
(u1) edge[bend left=30] (1)
(u1) edge[bend left=30] (2)
(w1) edge[bend left=40] (nu)
(w2) edge[bend left=40] (nu)
(t1) edge[bend right=10] (1)
(t2) edge[bend right=10] (1)
(t1) edge[bend right=10] (2)
(t2) edge[bend right=10] (2)
(t1) edge[bend left=10] (nu)
(t2) edge[bend left=10] (nu)
;
\draw[Green,->,ultra thick] (u1.120) .. controls (0.5,9) and (1.5,9) .. (u1.60);
\path[Green,->,ultra thick]
(u1) edge[red,-|,bend right=30] (v1)
(1) edge[red,-|,bend right=30] (v2)
(1) edge[bend left=40] (w1)
(u1) edge[bend right=40] (w2)
(2) edge[bend right=40] (w2)
(v1) edge[bend left=40] (rho)
(v2) edge[bend left=40] (rho)
;
\draw[decoration={brace,mirror,raise=10},decorate,thick] (-2,-5) --  (-1,-5);
\node at (-1.5,-6){\footnotesize $x_P\neq 00$}; 
\draw[decoration={brace,mirror,raise=10},decorate,thick] (-0.5,-5) -- (2,-5);
\node at (0.75,-6){\footnotesize $x_{v_1}=\neg x_{u_1}$}; 
\draw[decoration={brace,mirror,raise=10},decorate,thick] (2.5,-5) -- (5,-5);
\node at (3.75,-6){\footnotesize $x_{v_2}=\neg x_{p_1}$}; 
\draw[decoration={brace,mirror,raise=10},decorate,thick] (5.5,-5) -- (8,-5);
\node at (6.75,-6){\footnotesize $x_{w_1}=x_{p_1}$}; 
\draw[decoration={brace,mirror,raise=10},decorate,thick] (8.5,-5) -- (12.5,-5);
\node at (10.5,-6){\footnotesize $x_{w_2}=x_{u_1}\lor x_{p_2}$}; 
\draw[decoration={brace,mirror,raise=10},decorate,thick] (13,-5) -- (17,-5);
\node at (15,-6){\footnotesize $x_{\rho}=x_{v_1}\land x_{v_2}$}; 
\draw[decoration={brace,mirror,raise=10},decorate,thick] (17.5,-5) -- (20,-5);
\node at (18.75,-6){\footnotesize $x_{\rho}=x_\nu$}; 
\draw[decoration={brace,mirror,raise=10},decorate,thick] (-0.5,-6) -- (17,-6);
\node at (8.25,-7){$\omega$}; 
\draw[decoration={brace,mirror,raise=10},decorate,thick] (-2,-7) -- (20,-7);
\node at (9,-8){$\psi$}; 
\draw[rounded corners=5] (-4,4) rectangle (22,-9);
\node at (9,-9.5){\large $D_\psi\setminus \LAMBDA$}; 
\draw[rounded corners=5] (0,5) rectangle (18,15);
\node at (9,15.5){\large $C'$}; 
\end{tikzpicture}
\]
{\caption{\label{fig:D_Psi}
The SID $D_\Psi$ constructed from the circuit $h\in\functions(C)$ described in Figure~\ref{fig:h}. Black arrows represent \Null{} arcs. Braces correspond to sets of clauses contained in $\psi$, and below each brace the necessary and sufficient conditions for the corresponding clauses to be simultaneously satisfied by an assignment $x$ are given.}}
\end{figure}
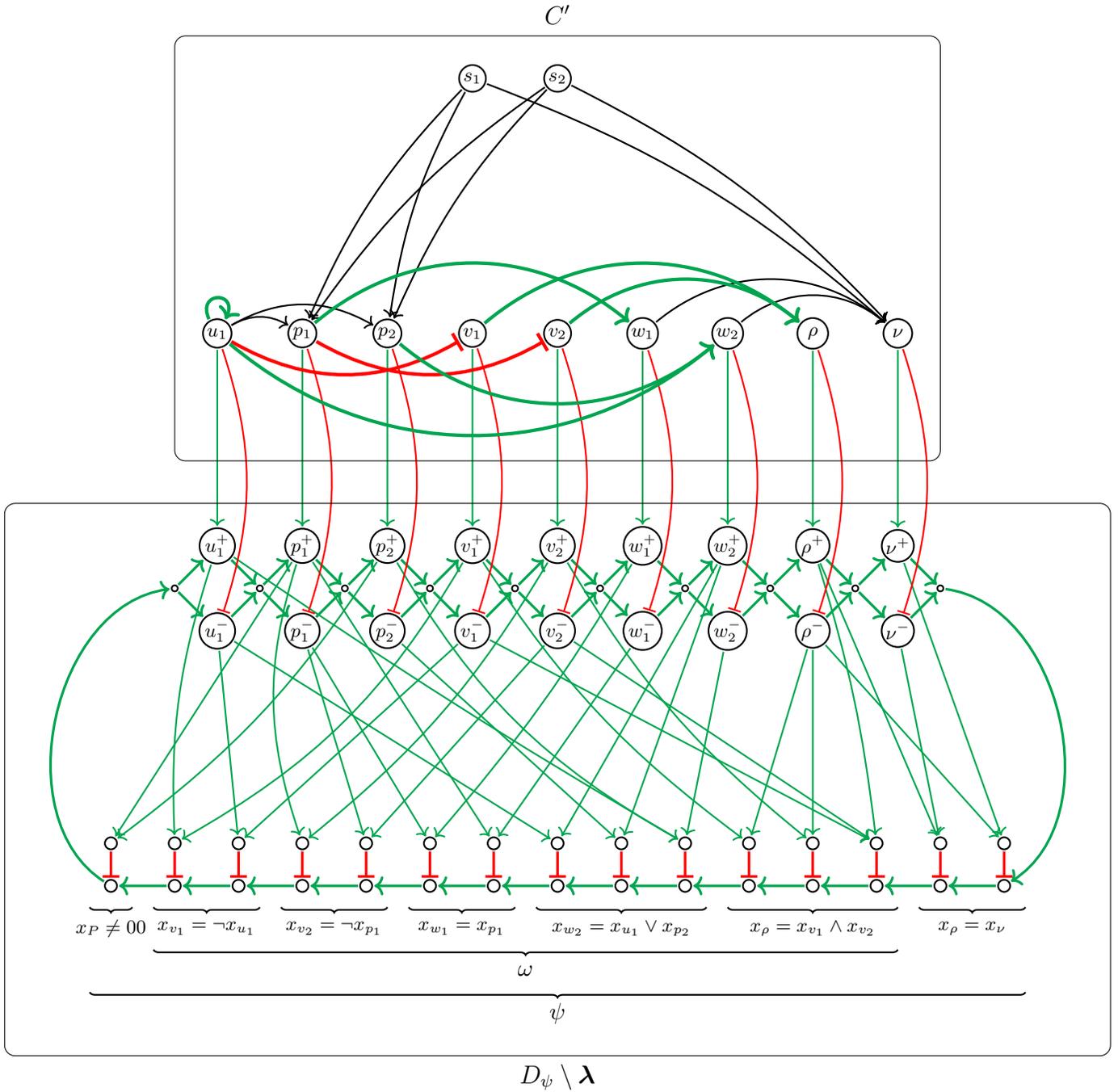

\medskip
Since $C'$ has $m$ positive loops and no other cycle, and since $D_{\Psi}\setminus  V_{C'}$ has a positive feedback vertex set of size one, we deduce that $D_{\Psi}$ has a positive feedback vertex set of size $m+1$. Thus, $\maxpf(D_{\Psi})\leq 2^{m+1}$ by the positive feedback bound. Actually, since $D_\Psi$ has $m+1$ vertex-disjoint positive cycles, we have $\tau^+(D)=m+1$ (since $C'$ has $m$ positive loops and $D_\psi\setminus \LAMBDA$ has a positive cycle). Putting things together, we prove that the positive feedback bound is reached if and only if $\Psi$ is satisfiable, and this proves Theorem~\ref{theorem:mfpp_unbounded}.

\begin{lemma}\label{lem:D_Psi}
$\maxpf(D_{\Psi})=2^{m+1}$ if and only if $\Psi$ is satisfiable. 
\end{lemma}

\begin{proof}
Suppose that $\Psi$ is satisfied by an assignment $\zeta\in\bool^{\GLAMBDA}$. By Lemma~\ref{lem:pro3} there is $h'\in\functions(C')$ such that $\psi$ is satisfied by all the $2^m$ fixed points of $h'$. Let $f\in\functions(D_\Psi)$ be defined as follows. First, $f(x)_{V_{C'}}=h'(x_{V_{C'}})$ for every configuration $x$ on $V_\Psi$. Second, for every $i\in V_\Psi\setminus V_{C'}$, we define $f_i$ has the AND function if $i\in\ELL$ and the OR function otherwise (vertices in $\ELL$ are those with two in-neighbors corresponding to the positive and negative literals associated with a variable in $\LAMBDA$). By the first part of the definition, the set of fixed points of $h'$ is the set of partial fixed points of $f$. So $\psi$ is satisfied by the $2^m$ partial fixed points of $f$, and we deduce from Lemma~\ref{lem:ext1} that $f$ has at least $2^{m+1}$ fixed points. Thus, $\maxpf(D_{\Psi})\geq 2^{m+1}$ and $\maxpf(D_{\Psi})=2^{m+1}$ by the positive feedback bound. 

\medskip
Conversely, suppose that $\maxpf(D_{\Psi})=2^{m+1}$. Let
$f\in\functions(D_\Psi)$ with $2^{m+1}$ fixed points. Let $h'$ be the BN with
component set $V_{C'}$ defined by $h'(x_{V_{C'}})=f(x)_{V_{C'}}$ for all
configurations $x$ on $V_\Psi$; there is no ambiguity since vertices in $C'$
have only in-neighbors in $C'$ and, thanks to this property, the SID of $h'$
is $D_\Psi\setminus U_\psi=C'$. Furthermore, the set of fixed points of $h'$
is the set of partial fixed points of $f$. Hence, $f$ has $2^m$ partial fixed
points. Since $f$ has at most two fixed points extending the same partial
fixed point, and since every fixed point of $f$ extends a partial fixed
point, we deduce that, for every partial fixed point $x$, $f$ has exactly two
fixed points extending $x$. Hence, by Lemma~\ref{lem:ext2}, $\psi$ is
satisfied by $x_{\LAMBDA}\xor\epsilon(f)$ for every fixed point $x$ of $h'$.
This is equivalent to say that $\psi$ is satisfied by every $\epsilon$-fixed
point of $h'$, where $\epsilon$ is any configuration on $V_{C'}$ extending
$\epsilon(f)$, and we deduce from Lemma~\ref{lem:pro3} that $\Psi$ is
satisfiable.
\end{proof}

\begin{remark}
By this lemma, the problem of deciding if $\maxpf(D)=2^{\tau^+(D)}$ is
$\NEXPTIME$-complete.  This is interesting since many works have been devoted
to study the conditions for similar bounds to be reached, see
\cite{GR11,GRF16,ARS17} and the references therein. In particular,
\cite{ARS17} gives a graph-theoretical characterization of the SIDs $D$ with
only positive arcs such that $\maxpf(D)=2^{\tau^+(D)}$, and we may ask if the
problem remains as hard under this restriction. Furthermore, the much studied
{\em network coding problem} in information theory can be restated as a
problem concerning BNs that consists in deciding if a bound similar to the
positive feedback bound is reached \cite{GR11}.  
\end{remark}

\begin{remark}
With slightly more precise arguments, we can prove that $\maxpf(D_\Psi)=2^{m+1}+\alpha$, where $\alpha$ is the maximum number of clauses contained in $\Psi$ that can be simultaneously satisfied. 
\end{remark}

\section{\MinFPP{}}\label{section:min}

In this section, we study decision problems related to the minimum number of fixed points, and we obtain the following tight complexity results. A new complexity class is involved: $\NPoNP$ (often denoted by $\sum^P_2$), which contains decision problems computable in polynomial time on a non-deterministic Turing machine with an oracle in $\NP$.

\begin{theorem} \label{theorem:minfpp}
Let $k\geq 1$ and $d\geq 2$ be fixed integers.
\begin{itemize}
\item
\kminfpp{k} and \minfpp{} are $\NEXPTIME$-complete.
\item
When $\Delta(D) \leq d$, \kminfpp{k} is $\NPoNP$-complete.
\item
When $\Delta(D) \leq d$, {\minfpp} is $\NPoSPoly$-complete.
\end{itemize}
\end{theorem}

We first prove the upper bounds.

\begin{lemma} \label{theorem:minfpp}
Let $k\geq 1$ and $d\geq 2$ be fixed integers. 
\begin{itemize}
\item
\kminfpp{k} and \minfpp{} are in $\NEXPTIME$.
\item
When $\Delta(D) \leq d$, \kminfpp{k} is in $\NPoNP$.
\item
When $\Delta(D) \leq d$, {\minfpp} is in $\NPoSPoly$.
\end{itemize}
\end{lemma}

\begin{proof}
As argued in Section~\ref{section:definition}, \kminfpp{k} and \minfpp{} are in $\NEXPTIME$. We now consider the restrictions to SIDs with bounded maximum in-degree. 
	
\medskip
For \kminfpp{k}, consider the algorithm that takes as input \an{} SID $D$ with $\Delta(D)\leq d$, and proceeds as follows. 
\begin{enumerate}
\item 
Choose non-deterministically a BN $f$ with component set $V_D$ such that, for all $i\in V_D$, the local function $f_i$ only depends on components in $N(i)$; this can be done in linear time since each local function $f_i$ can be represented using $2^{|N(i)|}\leq 2^d$ bits. 
\item 
Compute the SID $D_f$ of $f$; this can be done in quadratic time since, to compute the in-neighbors of each $i\in V_D$ and the corresponding signs, we only have to consider $2^{|N(i)|}\leq 2^d$ configurations (for each configuration $x$ on $N(i)$ and $j\in N(j)$, we compare $f_i(\tilde x)$ and $f_i(\tilde x\xor e_j)$ where $\tilde x$ is any configuration on $V_D$ extending $x$). 
\item 
Decide, with a call to the $\NP$-oracle, if $f$ has at least $k$ fixed points (this decision problem is trivially in $\NP$: choose non-deterministically $k$ distinct configurations, and accept if and only if they are all fixed points).
\item 
Accept if and only if the oracle's answer is no and $D_f=D$. 
\end{enumerate}
This non-deterministic polynomial time algorithm has an accepting branch if and only if $\minpf(D)<k$. Thus, when $\Delta(D)\leq d$, \kminfpp{k} is in $\NPoNP$. 

\medskip
For \minfpp{}, the algorithm we consider is the one described in Lemma~\ref{lemma:mfpp-Delta_bounded_in}, excepted that it accepts if and only if $\phi(f)<k$ and $D_f=D$. With this modification, we obtain a non-deterministic polynomial time algorithm which calls the $\SPoly$ oracle and has an accepting branch if and only if $\minpf(D)<k$. Thus, when $\Delta(D)\leq d$, $\minfpp{}$ is in $\NPoSPoly$.
\end{proof}

For the lower bounds, we use reductions based on the construction $D_\psi$ given in Section~\ref{section:max_k>=2}. We thus use the notations of that section, and we start with an adaptation of that construction suited for the study of the minimum number of fixed points.

\begin{definition}[$D^-_\psi$, extension of $D^-_\psi$]
We denote by $D^-_\psi$ the SID obtained from the SID $D_\psi$ of Definition \ref{def:D_psi} by making negative the arc $(c_1,\ell_0)$. An {\em extension} of $D^-_\psi$ is defined as in Definition \ref{def:extension} with $D^-_\psi$ instead of $D_\psi$. 
\end{definition}

As previously, given $f\in\functions(D)$ and setting $I=V_D\setminus U_\psi$, we say that a configuration $z$ on $I$ is a {\em partial fixed point} of $f$ if $f(x)_I=x_I=z$ for some configuration $x$ on $V_D$ (and since there is no arc from $U_\psi$ to $I$, if $z$ is a partial fixed point, then $f(x)_I=z$ for {\em every} extension $x$ of $z$).

\medskip
The following lemmas are adaptations of Lemmas~\ref{lem:ext1} and \ref{lem:ext2} to the above definition. Together, they show that $\minpf(D^-_\psi)=0$ if and only if $\psi$ is satisfiable. 

\begin{lemma} \label{lem:ext1_neg}
Let $D$ be an extension of $D^-_\psi$. Let $f\in\functions(D)$ such that, for all $i\in U_\psi$, $f_i$ is the AND function if $i\in\ELL$ and the OR function otherwise. Let $z$ be a partial fixed point of~$f$. Then $f$ has at most one fixed point extending $z$ and, if $\psi$ is satisfied by $z_{\LAMBDA}$, then $f$ has no fixed point extending $z$.
\end{lemma}

\begin{proof}
Let $I=V_D\setminus U_\psi$, and let $f^0,f^1$ be the BNs with component set $V_D$ defined as follows. First, $f^0_I=f^1_I=z$.  Second, $f^0_{\ell_0}=0$, $f^1_{\ell_0}=1$. Third, $f^0_i=f^1_i=f_i$ for all $i\in U_\psi\setminus \{\ell_0\}$. Then $f^0,f^1$ have the same SID, which is the SID obtained from $D$ by removing all the arc $(j,i)$ with $i\in I\cup\{\ell_0\}$, and which is thus acyclic. Hence, $f^0$ has a unique fixed point $x^0$ and $f^1$ has a unique fixed point $x^1$. 

\medskip
Let $D'$ obtained from $D$ by making positive the arc $(c_1,\ell_0)$, so that $D'$ is an extension of $D_\psi$. Let $f'\in\functions(D')$ such that $f'_i=f_i$ for all vertices $i\neq \ell_0$. By Lemma~\ref{lem:ext1}, $f'$ has a fixed point extending $z$, and another fixed point extending $z$ if $\psi$ is satisfied by $z_{\LAMBDA}$. 

\medskip
It is clear that if $x$ is a fixed point of $f$ or $f'$ extending $z$, then
$f^0(x)=x$ if $x_{\ell_0}=0$ and $f^1(x)=x$ if $x_{\ell_0}=1$, so $x$ is one
of $x^0,x^1$. Thus, the set of fixed points of $f$ extending $z$ is included
in $\{x^0,x^1\}$, and similarly for $f'$. Furthermore, if $f'(x)=x$ then
$f(x)\neq x$ (because $f_{\ell_0}(x)=\neg x_{c_1}$ and
$f'_{\ell_0}(x)=x_{c_1}$). Since $f'$ has a fixed point extending $z$, we
deduce that $f$ has at most one fixed point extending $z$ and, since $f'$
has two fixed points extending $z$ if $\psi$ is satisfied by $z_{\LAMBDA}$,
we deduce that $f$ has no fixed point extending $z$ if $\psi$ is satisfied by
$z_{\LAMBDA}$.
\end{proof}

Given an extension $D$ of $D^-_\psi$ and $f\in\functions(D)$, we define the assignment $\epsilon(f)$ exactly as previously: for $r\in [n]$, $\epsilon(f)_{\lambda_r}=0$ if $f_{\ell_r}$ is the OR function, and $\epsilon(f)_{\lambda_r}=1$ otherwise. 

\begin{lemma} \label{lem:ext2_neg}
Let $D$ be an extension of $D^-_\psi$. Let $f\in\functions(D)$ and let $z$ be a partial fixed point of $f$. If $f$ has no fixed point extending $z$, then $\psi$ is satisfied by $z_{\LAMBDA}\xor\epsilon(f)$. 
\end{lemma}

\begin{proof}
Let $f^0,f^1,f'$ be the BNs defined from $f$ as in the previous proof, and
let $x^0,x^1$ be the fixed points of $f^0,f^1$, respectively (which extend
$z$). For all vertices $i\neq\ell_0$ and all $a\in\{0,1\}$ we have
$f_i(x^a)=f'_i(x^a)=f^a_i(x^a)=x^a_i$. Suppose that $f$ has no fixed point
extending $z$. Then, for all $a\in\{0,1\}$ we have $\neg
x^a_{c_1}=f_{\ell_0}(x^a)\neq x^a_{\ell_0}$ and thus
$f'_{\ell_0}(x^a)=x^a_{c_1}=x^a_{\ell_0}$. Consequently, $x^0$ and $x^1$ are
distinct fixed points of $f'$ extending $z$ and, according to
Lemma~\ref{lem:ext2}, $\psi$ is satisfied by $z_{\LAMBDA}+\epsilon(f')$.
Since $\epsilon(f')=\epsilon(f)$, this proves the lemma.
\end{proof}

The following lemma shows that \kminfpp{1} is as hard as \kminfpp{k} for every $k\geq 2$. 

\begin{lemma}\label{lem:k_to_1}
Let $k\geq 2$ and let $D$ be any SID. Let $D'$ be the SID obtained from $D$ by adding $\lceil \log_2 k \rceil$ new vertices and a positive loop on each new vertex. Then 
\[
\minpf(D')<k \iff \minpf(D)=0.
\]
\end{lemma} 

\begin{proof}
Let $\ell=\lceil \log_2 k \rceil$ so that $k\leq 2^\ell$. Let $H$ be the SID
with $\ell$ vertices and a positive loop on each vertex. Then $\functions(H)$
contains a unique BN, which is the identity, thus $\minpf(H)=2^\ell$. Since
$D'$ is the disjoint union of $D$ and $H$, 
\[
\minpf(D')=\minpf(D)\cdot\minpf(H)=\minpf(D)\cdot 2^\ell.
\]  
Thus, $\minpf(D')=0<k$ if $\minpf(D)=0$, and $\minpf(D')\geq 2^\ell\geq k$ otherwise.
\end{proof}

We now prove that, for SIDs with a bounded maximum in-degree, \kminfpp{k} and \minfpp{} are $\NPoNP$-hard and $\NPoSPoly$-hard, respectively. The proof is very similar to that of Lemma~\ref{lemma:mfpp-Delta_bounded_hard}. The hardness of \kminfpp{k} is obtained with a reduction from the following decision problem, known to be $\NPoNP$-complete \cite[Theorem 17.10]{P94}.

\begin{quote}\decisionpb
	{\QSAT ~(\qsat)}
	{a 3-CNF formula $\psi$ over  $\LAMBDA=\{\lambda_1,\dots,\lambda_n\}$ and $s\in [n]$.}
	{does there exist $z'\in\bool^{\LAMBDA'}$, where $\LAMBDA'=\{\lambda_1,\dots,\lambda_s\}$, such that $\psi$ is satisfied by all the assignments $z\in\bool^{\LAMBDA}$ extending $z'$?}
\end{quote}

\begin{lemma} \label{lemma:minfpp-Delta_bounded_hard}
Let $k\geq 1$ and $d\geq 2$ be fixed integers.
\begin{itemize}
\item
When $\Delta(D)\leq d$, \kminfpp{k} is $\NPoNP$-hard.
\item
When $\Delta(D)\leq d$, \minfpp{} is $\NPoSPoly$-hard.
\end{itemize}
\end{lemma}

\begin{proof}
For the second item, we present a reduction from \emsat{}. Let $\psi$ be a CNF formula over the set of variable $\LAMBDA=\{\lambda_1,\dots,\lambda_n\}$. Let $s\in [n]$ and $\LAMBDA'=\{\lambda_1,\dots,\lambda_s\}$. Given $z'\in\bool^{\LAMBDA'}$, we denote by $E(z')$ the $2^{n-s}$ assignments $z\in\bool^{\LAMBDA}$ extending $z'$. Then we denote by $\alpha(z')$ the number of $z\in E(z')$ satisfying $\psi$, and $\alpha^*$ is the maximum of $\alpha(z')$ for $z'\in\bool^{\LAMBDA'}$. Thus, $(\psi,s)$ is a true instance if and only if $\alpha^*\geq 2^{n-s-1}$. 
	
\medskip	
Let $D_{\psi,s}$ be the extension of $D^-_\psi$ obtained by adding, for every $i\in\LAMBDA\setminus \LAMBDA'$, a positive loop on $i$. Let us prove that:
\[
\minpf(D_{\psi,s}) = 2^{n-s} -\alpha^*.
\]

\medskip
Let $z'\in\bool^{\LAMBDA'}$ such that $\alpha(z')=\alpha^*$. Let $f\in\functions(D_{\psi,s})$ such that $f_{\LAMBDA'}=z'$ and, for all $i\in U_\psi$, $f_i$ is the AND function if $i\in\ELL$ and the OR function otherwise. Since each vertex in $ \LAMBDA\setminus \LAMBDA'$ has a positive loop and no other in-coming arc, $E(z')$ is the set of partial fixed points of $f$. Hence, by Lemma~\ref{lem:ext1_neg}, for every $z\in E(z')$, $f$ has at most one fixed point extending $z$ and, if $\psi$ is satisfied by $z$, then $f$ has no fixed point extending $z$. Since every fixed point extends a partial fixed point, we deduce that $f$ has at most $2^{n-s}-\alpha(z')$ fixed points, and thus $\minpf(D_{\psi,s}) \leq  2^{n-s} -\alpha(z')=2^{n-s}-\alpha^*$.  

\medskip
For the other direction, let any $f\in\functions(D_{\psi,s})$ and set
$\epsilon=\epsilon(f)$. Since vertices in $\LAMBDA'$ are sources,
$f_{\LAMBDA'}=z'$ for some $z'\in\bool^{\LAMBDA'}$. Since each vertex in 
$\LAMBDA\setminus \LAMBDA'$ has a positive loop and no other in-coming arc,
$E(z')$ is the set of partial fixed points of $f$. By Lemma
\ref{lem:ext2_neg}, for every $z\in E(z')$, if $\psi$ is not satisfied by
$z\xor\epsilon$, then $f$ has a fixed point extending $z$. Thus, $f$ has at
least $2^{n-s}-\alpha(z'\xor\epsilon_{\LAMBDA'})$ fixed points, and since
$\alpha(z'\xor\epsilon_{\LAMBDA'})\leq\alpha^*$, we deduce that
$\minpf(D_{\psi,s}) \geq  2^{n-s} -\alpha^*$. 

\medskip
We have $\Delta(D_{\psi,s})\leq 2n$ but, using Lemma~\ref{lem:delta_reduction}, we can obtain from $D_{\psi,s}$ \an{} SID $D'_{\psi,s}$ with $\Delta(D'_{\psi,s})\leq 2$ and at most $4n+2nm+1$ vertices such that  
\[
\minpf(D'_{\psi,s})=\minpf(D_{\psi,s})=2^{n-s}-\alpha^*.
\]
(For that we use the fact, showed above, that there is $f\in\functions(D_{\psi,s})$ with $\minpf(D_{\psi,s})$ fixed points where $f_i$ is the OR function for every vertex $i$ of in-degree three.)

\medskip
Consequently, we have $\alpha^*\geq 2^{n-s-1}$ if and only if $\minpf(D'_{\psi,s})\leq 2^{n-s-1}$. Thus, $(\psi,s)$ is a true instance of \emsat{} if and only if $(D'_{\psi,s},k)$ is a true instance of \minfpp{}, where $k= 2^{n-s-1}+1$ if $s<n$ and $k=1$ otherwise. Thus, when $\Delta(D)\leq d$, \minfpp{} is $\NPoSPoly$-hard.

\medskip
An other consequence is that $\minpf(D'_{\psi,s}) = 0$ if and only if $\alpha^*=2^{n-s}$, that is, there is a partial assignment $z'$ of the variables in $\LAMBDA'$ such that $\psi$ is satisfied by all the assignments extending $z'$. Thus, $(\psi,s)$ is a true instance of \qsat{} if and only if $D'_{\psi,s}$ is a true instance of  \kminfpp{1}, and therefore \kminfpp{1} is $\NPoNP$-hard. We then deduce from Lemma~\ref{lem:k_to_1} that, when $\Delta(D)\leq d$, \kminfpp{k} is $\NPoNP$-hard for all $k\geq 2$.
\end{proof}

It remains to prove that, in the general case, \kminfpp{k} and \minfpp{} are
$\NEXPTIME$-hard. We proceed with a reduction from \SSAT{}, which is very
similar to the one given in Section \ref{subsection:non borne} to obtain the
hardness of \maxfpp{}. We thus use the notations from that section. We consider
a succinct representation of a 3-CNF formula $\Psi$ with a set $\GLAMBDA$ of
$2^n$ variables (indexed by configurations on $W$) and with a set $\M$ of $2^m$
clauses (indexed by configurations on $U$). The following is an adaptation of
the construction $D_\Psi$ suited for the study of the minimum number of fixed
points.

\begin{definition}[$D^-_\Psi$]
We denote by $D^-_\Psi$ the SID obtained from the SID $D_\Psi$ of Definition \ref{def:D_Psi} by making negative the arc $(c_1,\ell_0)$.
\end{definition}

The main property is the following. The proof is almost identical to that of Lemma~\ref{lem:D_Psi}, using Lemmas \ref{lem:ext1_neg} and \ref{lem:ext2_neg} instead of Lemmas~\ref{lem:ext1} and \ref{lem:ext2}.

\begin{lemma}
$\minpf(D^-_{\Psi})=0$ if and only if $\Psi$ is satisfiable. 
\end{lemma}

\begin{proof}
Suppose that $\Psi$ is satisfied by an assignment $\zeta\in\bool^{\GLAMBDA}$. By Lemma~\ref{lem:pro3} there is $h'\in\functions(C')$ such that $\psi$ is satisfied by all the $2^m$ fixed points of $h'$. Let $f\in\functions(D_\Psi)$ be defined as follows. First, $f(x)_{V_{C'}}=h'(x_{V_{C'}})$ for every configuration $x$ on $V_\Psi$. Second, for every $i\in V_\Psi\setminus V_{C'}$, we define $f_i$ has the AND function if $i\in\ELL$ and the OR function otherwise (vertices in $\ELL$ are those with two in-neighbors corresponding to the positive and negative literals associated with a variable in $\LAMBDA$). By the first part of the definition, the set of fixed points of $h'$ is set of partial fixed points of $f$. So $\psi$ is satisfied by any partial fixed point $x$ of $f$, and we deduce from Lemma~\ref{lem:ext1_neg} that $f$ has no fixed point extending $x$. Since every fixed point extends a partial fixed point, we deduce that $f$ has no fixed point and thus $\minpf(D^-_{\Psi})=0$. 

\medskip
Conversely, suppose that $\maxpf(D_{\Psi})=0$. Let $f\in\functions(D_\Psi)$
without fixed point. Let $h'$ be the BN with component set $V_{C'}$ defined
by $h'(x_{V_{C'}})=f(x)_{V_{C'}}$ for all configurations $x$ on $V_\Psi$;
there is no ambiguity since vertices in $C'$ have only in-neighbors in $C'$
and, thanks to this property, the SID of $h'$ is $D_\Psi\setminus U_\psi=C'$.
Furthermore, the set of fixed points of $h'$ is the set of partial fixed
points of $f$. Let $x$ be a fixed point of $h'$. Since $f$ has no fixed point
extending $x$, by Lemma~\ref{lem:ext2_neg}, $\psi$ is satisfied by
$x_{\LAMBDA}\xor\epsilon(f)$.
This is equivalent to say that $\psi$ is satisfied by every $\epsilon$-fixed
point of $h'$, where $\epsilon$ is any configuration on $V_{C'}$ extending
$\epsilon(f)$, and we deduce from Lemma~\ref{lem:pro3} that $\Psi$ is
satisfiable.
\end{proof}

\begin{remark}
With slightly more precise arguments, we can prove that $\minpf(D^-_\Psi)=2^{m+1}-\alpha$, where $\alpha$ is the maximum number of clauses contained in $\Psi$ that can be simultaneously satisfied. 
\end{remark}

The following hardness results, the last we need to obtain, are immediate consequences.  

\begin{lemma}
\minfpp{} and \kminfpp{k} for every $k \geq 1$, are $\NEXPTIME$-hard.
\end{lemma}

\begin{proof}
By the previous lemma, $D^-_\Psi$ is a true instance of \kminfpp{1} if and only if $(h,C)$, the succinct representation of $\Psi$, is a true instance of \SthreeSAT{}. Therefore, \kminfpp{1} is $\NEXPTIME$-hard. We then deduce from Lemma~\ref{lem:k_to_1} that \kminfpp{k} is $\NEXPTIME$-hard for all $k\geq 2$. Consequently, \minfpp{} is also $\NEXPTIME$-hard.
\end{proof}

\section{Conclusion and perspectives}\label{sec:conclu}

In this paper, we studied the algorithmic complexity of many decision problems related to counting the number of fixed points of a BN from its SID only. Except for \kmaxfpp{1}, we proved exact complexity bounds for each one of these problems,
revealing a large range of complexities, some classes having a pretty scarce literature.

\medskip
The function problems of computing $\maxpf(D)$ or $\minpf(D)$ can, quite classically, be seen as $n$ (since the result ranges from $0$ to $2^n$, but the case $2^n$ can be treated separately) decision problems providing the bits of the answer by binary search. Table~\ref{table:classe_complexite} gives the worst case complexity of computing each of these bits in the different cases (minimum/maximum, degree bounded/unbounded). Note that, even though they do not intuitively correspond to counting problems, computing $\maxpf(D)$ or $\minpf(D)$ are proven to be $\SPoly$-hard problems even for bounded degree, from the proofs of Lemmas~\ref{lemma:mfpp-Delta_bounded_hard} and~\ref{lemma:minfpp-Delta_bounded_hard}.

\medskip
Even if the problems we studied are natural, they have not been considered
before. Indeed, most of the works on the complexity of BNs are of the form:
does a BN $f$ satisfies a given dynamical property $P$? Here, we studied
problems of the form: does \an{} SID $D$ admits a BN $f \in \functions(D)$
satisfying a given property $P$? In some cases, this new problem can be easier.
For instance, the problem of deciding if a BN (encoded as a concatenation of
local functions in conjunctive normal form) has at least one fixed point is
$\NP$-complete~\cite{K08}, but we showed that deciding if there exists a BN $f
\in\functions(D)$ with at least one fixed point is in $\Poly$. Conversely, the
problem of deciding if a BN $f$ has no fixed point is in $\coNP$, whereas it is
$\NEXPTIME$-complete to decide if there is a BN $f \in \functions(D)$ without
fixed point. 

\medskip
This new theme opens many further investigations. First, we proved that \kmaxfpp{1} is equivalent (up to a polynomial reduction) to the problem of finding an even cycle in a digraph. This problem is known to be in $\Poly$~\cite{McC04,RST99} but, to the best of our knowledge, no work has been done to show its $\mathsf C$-hardness for any complexity class $\mathsf C$. It would be interesting to find such lower bounds, because this problem is equivalent to many other decision problems of graph theory (\cite{RST99} lists several of them).

\medskip
Furthermore, we could study the effect of other restrictions on the SIDs considered. For instance, what happens for the family of SIDs with positives arcs only? It is an interesting problem, because the BN having such SIDs are monotone, and since Tarski~\cite{tarski1955lattice} monotone networks received great attention~\cite{ARS17,robert1985connection,alon1985asynchronous,aracena2004fixed,melliti2016asynchronous,just2008extremely,just2013cooperative1,just2013cooperative2}. We hope to be able to adapt our constructions in order to fit this new constraint. We conjecture that, in this case, there is an integer $k_0$ such that \kmaxfpp{k} is in $\Poly$ if $k \leq k_0$, and $\NP$-complete otherwise.

\medskip
As further variations, we could study unsigned interaction digraph, or consider
{\em automata network} instead of Boolean network (where a component can take
more than two states), or consider BNs that may ignore some arcs of the SID to
get closer to real experimental conditions (where arcs of the SID are not
always a hundred percent accurate). This could affect the problems'
difficulties drastically, and may reveal relations to \emph{network coding}
problems in information theory~\cite{li2003linear,GR11,GRF16}.

\medskip
Finally, we could think of new problems where, given \an{} SID $D$, we want to decide other properties shared by all $f \in \functions(D)$. Indeed, there are many other interesting BN properties, such as the number or the size of the limit cycles, the size of their transients\dots{} we could also look for properties of their basins of attraction, some complexity results already being known for threshold networks~\cite{floreen1993attraction}.

\medskip
As mentioned in the introduction, there are many update schedules that can be used to describe a dynamics from $f$, the most classical ones being the parallel, block-sequential, and asynchronous update schedules \cite{R86,AGMS09,R19}. With the parallel (or synchronous) update schedule, the dynamics is described by the digraph with an arc from $x$ to $f(x)$ for all configurations $x$, as in Example~\ref{ex:f}. For the other schedules, we need a notation. Let $V$ be the component set of $f$ and $I\subseteq V$. Let $f^I(x)$ be the configuration $y$ on $V$ such that $y_i=f_i(x)$ for $i\in I$, and $y_i=x_i$ for $i\not\in I$. Hence,  $f^I(x)$ is obtained from $x$ by updating in parallel components in $I$. A block-sequential update schedule  is then an ordered partition $I_1,\dots,I_k$ of $V$ and the resulting dynamics is the parallel dynamics of the BN $h=f^{I_k}\circ\dots\circ f^{I_2}\circ f^{I_1}$. Finally, the asynchronous dynamics of $f$ is described by the digraph with an arc from $x$ to $f^{\{i\}}(x)$ for all configurations $x$ and components $i$. In every case, attractors correspond to the terminal strongly connected components (and are also called limit cycles in the first two cases). In each case, the fixed points of $f$ correspond to attractors of size one and are, in this sense, invariant~\cite{GM90}. But the other attractors, and their basins, can change a lot according to the update schedule used~\cite{DNS12}. Hence, when studying dynamical properties not involving only fixed points, one has to fix the update schedule (but any deterministic dynamics can be obtained by a network updated in parallel, which can be constructed in polynomial time~\cite{R86}; hence in deterministic cases such as block-sequential, it may be more meaningful to quantify over the possible update schedules within the class under consideration, in the problem definition), and complexity results, as well as proof technics, should highly depend on the choice~\cite{ps21,chkpt20,pkch20,bgmps21}.

\section{Acknowledgments}

The authors would like to thank for their support
the \emph{Young Researcher} project ANR-18-CE40-0002-01 ``FANs'',
project ECOS-CONICYT C16E01,
and project STIC AmSud CoDANet 19-STIC-03 (Campus France 43478PD).

\bibliographystyle{plain}
\bibliography{sources}

\end{document}